\documentclass[10pt]{amsart}
\usepackage{a4}
\usepackage{amssymb}
\usepackage{amscd}
\usepackage{array}
\usepackage{multirow}
\usepackage{hhline}
\usepackage[all]{xy}

\newtheorem{thm}{Theorem}

\newtheorem{lem}[thm]{Lemma}
\newtheorem{prop}[thm]{Proposition}
\newtheorem{defn}[thm]{Definition}

\newtheorem{rem}[thm]{Remark}

\numberwithin{thm}{section}
\numberwithin{equation}{section}
\newcommand{\norm}[1]{\left\Vert#1\right\Vert}
\newcommand{\abs}[1]{\left\vert#1\right\vert}

\newcommand{\la}{\langle}
\newcommand{\ra}{\rangle}

\newcommand{\Gb}{\mathbb{G}}
\newcommand{\Hi}{\mathcal{H}}

\newcommand{\M}{\mathcal{M}}
\newcommand{\m}{\textbf{mod}}
\newcommand{\n}{\mathbb{N}}

\newcommand{\om}{\omega}

\newcommand{\prt}{\widehat{\otimes}}

\begin{document}
\title{Projectivity of modules over Fourier algebras}
\author{Brian E. Forrest, Hun Hee Lee, Ebrahim Samei}

\address{Brian E. Forrest: Department of Pure Mathematics, Faculty of Mathematics, University of Waterloo,
200 University Avenue West, Waterloo, Ontario, Canada N2L 3G1}
\address{Hun Hee Lee: Department of Mathematics, Chungbuk National University, 410 Sungbong-Ro, Heungduk-Gu, Cheongju 361-763, Korea}
\address{Ebrahim Samei: Department of Mathematics and Statistics, University of Saskatchewan, 106 Wiggins Road Saskatoon, SK S7N 5E6 CANADA}

\begin{abstract}
In this paper we will study the homological properties of various natural modules associated
to the Fourier algebra of a locally compact group. In particular, we will focus on the question of 
identifying when such modules will be projective in the category of operator spaces.
We will show that projectivity often implies that the underlying group is discrete and give evidence
to show that amenability also plays an important role.
\end{abstract}

\maketitle

\section{Introduction}
Over the past thirty years there has been a rich body of work dedicated to understanding the 
cohomological properties of the various Banach algebras arising in the study of locally compact 
groups. The seminal paper in this respect is certainly Barry Johnson's memoir \cite{John} in which he introduces
the concept of amenability for a Banach algebra, and shows that for a locally compact group 
$G$ the group algebra $L^{1}(G)$ is amenable if and only if $G$ is amenable in the classical 
sense. 

While applications of cohomology to harmonic anlysis may be more celebrated, there have also been
a number of significant studies related to the homological properties of the algebras arising 
from groups. Recently, H. G. Dales and M. E. Polyakov \cite{DP04} gave a detailed study of the homological 
properties of modules over the group algebra of a locally compact group. In their work
they focused primarily on the question of whether or not certain natural left $L^1(G)$-modules are respectively 
\textit{projective}, \textit{injective}, or \textit{flat}. They were able to show, for example, that 
when viewed as the dual left-module of $L^{1}(G)$, $L^{\infty}(G)$ is projective precisely when the 
group $G$ is finite. In stark contrast, they prove that $L^{\infty}(G)$  is always injective. They also showed that the measure algebra $M(G)$
is projective precisely when $G$ is discrete, while in this case injectivity is equivalent to the group $G$ being 
amenable. 

When $G$ is abelian the classical Fourier transform identifies $L^{1}(G)$ with a commutative Banach algebra 
$A(\widehat {G})$ of continuous functions on the dual group $\widehat G$ called the Fourier algebra. The Banach space dual of $A(\widehat {G})$
which, we denote by $VN(\widehat {G})$, can be identified with $L^{\infty}(G)$ in the obvious way. Again for abelian $G$, 
the Fourier-Steiltjes transform identifies $M(G)$ with another commutative Banch algebra $B(\widehat {G})$, called the Fourier-Stieltjes algebra. 
Duality tells us that for $G$ abelian, $VN(\widehat {G})$ is projective as a left $A(\widehat {G})$-module if and only if 
$\widehat {G}$ is finite, and that $B(\widehat {G})$ 
will be a projective  left $A(\widehat {G})$-module, precisely when the dual group $\widehat {G}$ is compact.  

If $G$ is not abelian, then the classical method for defining the Fourier and Fourier-Steiltjes algebras is 
no longer available. However, in \cite{Eymard}, Pierre Eymard succeeded in defining the Fourier algebra $A(G)$ and 
the Fourier-Stieltjes algebra $B(G)$ for any locally compact group $G$. Just as was the case for an abelian group, 
$A(G)$ and $B(G)$ are both commutative Banach algebras of continuous functions on $G$. In this paper, we will follow 
the lead set by Dales and Polyakov and 
study the homological properties of various natural $A(G)$-modules when $G$ is a locally compact, not necessarily abelian group. 
However, in our case, we will focus only on questions involving projectivity. 

We close this section by noting that as the predual of a von Neumann algebra, $A(G)$ carries the additional structure of
an operator space. Ample evidence exists to suggest that in studying the cohomological properties of $A(G)$, the operator structure 
plays a crucial role when $G$ is not abelian. Indeed, it is known that the Fourier algerba $A(G)$ is amenable as a 
Banach algebra if and only if the group $G$ is a finite extension of an abelian group \cite{ForRun}. This means that the dual to 
Johnson's fundamental theorem on amenability of $L^{1}(G)$, namely that $A(G)$ is amenable as a Banach algebra if 
and only if $G$ is amenable, fails spectacularly. In contract, Ruan showed that $A(G)$ is amenable in the operator 
space category precisely when $G$ is amenable \cite{Ru}. Ruan's result can be viewed as a true analog of Johnson's Theorem 
because for group algebras the underlying operator space structure is essentially trivial and as such amenability 
in the category of Banach spaces and in the category of operator spaces is the same.

The operator space structure for $A(G)$ has also been shown to play an important role in the homology of $A(G)$. 
Recall, that a Banach algebra $\mathcal{A}$ is called \textit{biprojective} if it is projective 
 when viewed as a bimodule over itself.  For example, 
it is a well known result of Khelemskii (see \cite{Hel}) that $L^{1}(G)$ is \textit{biprojective} if and only if $G$ is compact.
Duality would suggest that the Fourier algebra would be biprojective if and only if $G$ is discrete. 
However, if $G$ is an amenable discrete group, then because biprojective Banach algebras with bounded approximate 
identities are amenable, $A(G)$ is biprojective precisely when $G$ has an abelian 
subgroup of finite index. Moreover, if $\mathbb{F}_2$ is the free group on two generators, then $A(\mathbb{F}_2)$ is not biprojective (\cite{Stei}).
In contrast, given the obvious operator space analog of the notion of biprojectivity, it can be shown that $A(G)$ is operator 
biprojective if and only if $G$ is discrete (see \cite{Ar02} and \cite{W02}). In this paper, we will make heavy use of the 
operator space structure of $A(G)$ and its natural modules to establish analogs of the projectivity results for  
all of the left $L^{1}(G)$-modules studied by Dales and Polyakov, as well as some additional natural left $A(G)$-modules.    
   
\section{Preliminaries}

Let $G$ be a locally compact group with a fixed left Haar measure $dx$. Let $L^{1}(G)$ denote 
the group algebra of $G$. Then $L^{1}(G)$ is an involutive Banach algebra under convolution.

Let $\Sigma _{G}$ be the collection of all equivalence classes of weakly continuous
unitary representations of $G$ into $B(\mathcal{H}_{\pi })$ for some Hilbert
space $\mathcal{H}_{\pi }.$ 

Every $\pi \in \Sigma _{G}$ lifts to a $^*$-representation of $L^{1}(G)$ via the formula
\[\langle \pi (f) \xi ,\eta \rangle =\int_{G} f(x) \langle \pi (x) \xi ,\eta  \rangle dx.\]

We can define a norm on $L^{1}(G)$ by 

\[\| f\| _{*}=\sup \{\| \pi (f)\| : \pi \in \Sigma_G \}.\]

The completion of $L^{1}(G)$ with respect to this nom is the group $C^*$-algebra, which we denote by 
$C^{*}(G)$.  

A function of the form 
\[
u(x)= \la \pi (x)\zeta ,\eta \ra (\zeta ,\eta \in \mathcal{H})\]
is called a coefficient functions of $\pi .$ Let 
\[
B(G)=\{u(x)= \langle\pi (x)\xi ,\eta \rangle: \pi \in \Sigma_{G},\xi ,\eta \in \mathcal{H}_{\pi }\}. 
\]
Then $B(G)$ is a commutative algebra of continuous functions on $G$ with
respect to pointwise operations. It is well known that $B(G)$ can be
identified with the dual of the group $C^{*}$-algebra $C^{*}(G)$. In case $G$ is abelian, $C^*(G)$ is the image of
$C_0(\widehat {G})$ under the generalized Fourier transform. 

 With
respect to the dual norm, $B(G)$ is a commutative Banach algebra called the
Fourier-Stieltjes algebra of $G.$

The left regular representation $\lambda $ acts on $L^{2}(G)$ as follows: 
\[
(\lambda (y)(f))(x) := f(y^{-1}x) 
\]
for each $x,y\in G,$ $f\in L^{2}(G).$ We denote by $VN(G)$ the closure of
$\text{span}\{\lambda (x): x\in G\}$ in the weak operator topology of $B(L^{2}(G)).$
$VN(G)$ is a von Neumann algebra called the group von Neumann
algebra of $G$. Its predual is $A(G),$ the algebra of continuous functions
that are coefficient functions of the left regular representation $\lambda $
of $G$. $A(G)$, the Fourier algebra of $G,$ is a closed ideal in $B(G)$ and
the norm induced on $A(G)$ as the predual of $VN(G)$ agrees with the norm it
inherits from $B(G)$ \cite{Eymard}.

It will be important for our purposes to note that while $L^{1}(G)$ always has a bounded 
approximate identity (b.a.i), Leptin \cite{Lep} showed that $A(G)$ has a b.a.i. if and only if
$G$ is amenable. 

Let $C^*_r(G)$ denoted the norm closure of $\{\lambda (f): f\in L^{1}(G)\}$ in the space 
$B(L^{2}(G))$. $C^*_r(G)$ is called the reduced $C^{*}$-algebra of $G$. When $G$ is amenable, 
$C^*(G)$ and $C^{*}_r(G)$ agree, but if $G$ is not amenable, then $C^{*}_r(G)$ is a proper quotient of 
$C^*(G)$.

We let $C^*_\delta (G)$ denote the $C^*$-algebra in $B(L^{2}(G))$ generated by 
$\{\lambda (x): x\in G\}$. If $G$ is discrete the clearly $C^*_\delta (G)=C^*_r(G)$. 
However for non-discrete groups these algebras are different. When $G$ is abelian,  
$C^*_\delta (G)$ is the image of $AP(\widehat {G})$, the algebra of almost periodic functions
on $\widehat {G}$. 

We will need to know that if $H$ is an open subgroup of $G$, then there is a natural 
completely isometric injection $i:C^*(H)\rightarrow C^*(G)$ (see \cite[Section 5]{BKLS}). 

It is clear that $VN(G)$ is a commutative $A(G)$-bimodule under the dual action 
\[UCB(\widehat {G})=\bar {\overline{\text{span}}\{u\cdot T : u\in A(G),T\in VN(G)\}}\subseteq VN(G).\]
Then $UCB(\widehat {G})$ is called the space of \textit{uniformly continuous functionals on} 
$A(G)$.
It is clear that $UCB(\widehat {G})$ is an $A(G)$-submodule of $VN(G)$. 

We will now briefly remind the reader about the basic properties of operators spaces.
We refer the reader to \cite{ER00} for further details concerning the notions presented below. 

Let $\mathcal{H}$ be a Hilbert space. Then there is a natural identification between the space
$M_n(B(\mathcal{H}))$ of $n\times n$ matrices with entries in $B(\mathcal{H})$ and the space 
$B(\mathcal{H}^n)$. This allows us to define a sequence of norms $\{\| \cdot \| _n\}$ on the spaces
$\{M_n(B(\mathcal{H}))\}$. If $V$ is any subspace of $B(\mathcal{H})$, then the spaces 
$M_n(V)$ also inherit the above norm. A subspace $V\subseteq B(\mathcal{H})$ together with the family 
$\{\| \cdot \| _n\}$ of norms on $\{M_n(V)\}$ is called a \textit{concrete operator space}. This 
leads us to the following abstract definition of an operator space: 

\begin{defn}
An operator space is a vector space $V$ together with a family $\{\norm{\cdot}_{n}\}$ of Banach space norms on $M_{n}(V)$ such that for each
$A\in M_{n}(V),B\in M_{m}(V)$ and $[a_{ij}],[b_{ij}]\in M_{n}(\mathbb{C})$

\[
\begin{array}{ll}
i) & \parallel \left[ 
\begin{array}{ll}
A & 0 \\ 
0 & B
\end{array}
\right] \parallel _{n+m}=\max \{\parallel A\parallel _{n},\parallel
B\parallel _{m}\} \\ 
& \\
ii) & \parallel [a_{ij}]A[b_{ij}]\parallel _{n}\leq \parallel
[a_{ij}]\parallel \parallel A\parallel _{n}\parallel [b_{ij}]\parallel _{n}%
\end{array}
\]

Let $V,W$ be operator space, $\varphi :V\rightarrow W$ be linear. Then
\[
\parallel \varphi \parallel _{cb}=\sup_n \{\parallel \varphi _{n}\parallel \}
\]
where $\varphi _{n}:M_{n}(V)\rightarrow M_{n}(W)$ is given by 
\[
\varphi _{n}([v_{ij}])=[\varphi (v_{ij})].
\]

We say that $\varphi $ is completely bounded if $\parallel \varphi \parallel_{cb}<\infty ;$
is completely contractive if $\parallel \varphi \parallel_{cb}\leq 1$ and is a complete isometry if each $\varphi _{n}$ is an isometry.

Given two operator spaces $V$ and $W$, we let $CB(V,W)$ denote the space of all completely 
bounded maps from $V$ to $W$. Then $CB(V,W)$ becomes a Banach space with respect to the norm 
$\| \cdot \| _{cb}$ and is in fact an operator space via the identification $M_n(CB(V,W))\cong 
CB(V,M_n(W))$.  
\end{defn}

It is well-known that every Banach space can be given an operator 
space structure, though not necessarily in a unique way. It is also 
clear that any subspace of an operator space is also an operator space with respect to the inherited norms.
Moreover, for duals and preduals of operator spaces, there are canonical operator space structures.
As such the predual of a von Neumann algebra
and the dual of a $C^*$-algebras respectively, the Fourier and Fourier-Stieltjes algebras 
inherit natural operator space structures.  

Given two Banach spaces $V$ and $W$, there are many ways to define a norm on the 
algebraic tensor product $V\otimes W$. Distinguished amongst such norms is the \textit{Banach space 
projective tensor product norm} which we denote by $V\otimes ^\gamma W$. A fundamental property of 
the projective tensor product is that there is a natural isometry between $(V\otimes ^\gamma W)^*$ and 
$B(V,W^*)$. Given two operator spaces $V$ and $W$, there is an operator space analog of the projective
tensor product norm which we denote by $V\widehat{\otimes}\,W$. In this case, we have a natural 
complete isometry between $(V\widehat{\otimes}\,W)^*$ and $CB(V,W^*)$. 

The operator analog of the injective tensor product will be denoted by $V\otimes_{\min}W$. 

\begin{defn}

A Banach algebra $\mathcal{A}$ that is also an operator space is called a 
\textit{completely contractive Banach algebra} if the multiplication map 

\[m:\mathcal{A}\widehat{\otimes}\,\mathcal{A}\rightarrow \mathcal{A}\]

with 
\[m(u\otimes v)=uv,\]
is completely contractive. In particular, both $B(G)$ and $A(G)$ are 
completely contractive Banach algebras (see \cite{Eymard}).  

Let $\mathcal{A}$ be a completely contractive Banach algebra. An operator space $X$ is called a 
\textit{completely bounded left $\mathcal{A}$-module}, if $X$ is a left $\mathcal{A}$-module and if 
\[\pi _{X} :\mathcal{A}\widehat{\otimes}\, X\rightarrow X\]

with 
\[\pi _{X}(u\otimes x)=u\cdot x,\]
is completely bounded. 

We say that $X$ is \textit{essential} if $\mathcal{A}\cdot X$ is dense in $X$. 

We will let the collection of all completely bounded left $\mathcal{A}$-modules be denoted 
by $\mathcal{A}$-\textbf{mod}. 

If $X,Y\in \mathcal{A}$-\textbf{mod}, then we let $_{\mathcal{A}}CB(X,Y)$ denote the space of all 
completely bounded left $\mathcal{A}$-module maps from $X$ to $Y$. 

We can define a \textit{completely bounded right $\mathcal{A}$-module}
and a \textit{completely bounded $\mathcal{A}$-bimodule} analogously. The collection of all such modules
 will be denoted by \textbf{mod}-$\mathcal{A}$ and $\mathcal{A}$-\textbf{mod}-$\cdot \mathcal{A}$ respectively.

\end{defn}

In general, if $\mathcal{A}$ is a completely contractive Banach algebra, then its dual 
space $\mathcal{A}^*$ is a completely bounded left $\mathcal{A}$-module via the action 
\[(u\cdot T)(v)=T(vu)\]
for every $u,v\in \mathcal{A}$ and $T\in \mathcal{A}^*$. Moreover, every closed left 
$\mathcal{A}$-submodule $Y$ of $\mathcal{A}^*$ is also a completely bounded left $\mathcal{A}$-module.

If $X$ is a completely bounded left $\mathcal{A}$-module and $Y$ is an operator space, then 
$X\widehat{\otimes}\, Y$ becomes a completely bounded left $\mathcal{A}$-module via
\[u\cdot (x\otimes y)=ux\otimes y.\]
Similarly, if $Y$ is a completely bounded right $\mathcal{A}$-module and $X$ is an operator space, then 
$X\widehat{\otimes}\, Y$ becomes a completely bounded right $\mathcal{A}$-module via
\[(x\otimes y)\cdot u=x\otimes yu.\]

If $X$ and $Y$ are completely bounded right and left $\mathcal{A}$-modules, respectively, then we define 
\[X\widehat{\otimes}\,_{A}Y=(X\widehat{\otimes}\,Y)/N\]
where $N$ is the closed subspace of $X\widehat{\otimes}\,Y$ spanned by elements of the form
\[xu\otimes y-x\otimes uy.\]

Finally, we will need the following operator space analog of Grothendieck's classical approximation property.
The definition of the operator approximation property is due to Effros and Ruan \cite{ER00}. The definition 
we give is not the original statement of the property but rather has been established as an 
equivalent formulation in the same reference.

\begin{defn}\label{def-OAP}
We say that an operator space $X$ has the \textit{Operator Approximation Property} (OAP) if
the natural map 
\[J:V^*\widehat{\otimes}\,V\rightarrow V^*\otimes _{min}V\]
is one-to-one.

\end{defn}

\section{General theory related to projectivity}

In this section we will establish some basic properties of projective modules that we will need in our study. 

\begin{defn} 
Let $\mathcal{A}$ be a completely contractive Banach algebra. Let $X,Y$ be completely bounded 
left $\mathcal{A}$-modules. A map $T\in CB(X,Y)$ is \textit{admissible} if $ker T$ is completely 
complemented in $X$ and the range of $T$ is closed and completely complemented in $Y$. 

Let $X\in \mathcal{A}$-{\bf mod}. Then $X$ is said to be \textit{operator projective in} 
$\mathcal{A}$-{\bf mod} if whenever $E,F\in \mathcal{A}$-{\bf mod} and 
$T\in _{\mathcal{A}}CB(E,F)$ is admissible and surjective, 
then for each $S\in  _{\mathcal{A}}CB(X,F)$, there exists $R\in  _{\mathcal{A}}CB(X,E)$ such that 
$T\circ R=S$. 

\begin{equation*}
\begin{tabular}{lll}
&  & $X$ \\ 
& $\overset{R}{\swarrow }$ & $\downarrow S$ \\ 
$E$ & $\underset{T}{\longrightarrow }$ & $F$%
\end{tabular}
\end{equation*}

Equivalently, $X$ is operator projective in $\mathcal{A}$-{\bf mod} if $Z$ is any 
submodule of $F$, then every $T\in _{\mathcal{A}}CB(X,F/Z)$ lifts to 
a map in $T\in _{\mathcal{A}}CB(X,F)$

We can also define operator projectivity in ${\bf mod}$-$\mathcal{A}$ and in $\mathcal{A}$-${\bf mod}$-$\mathcal{A}$ 
in a similar manner. In this case, we say that $\mathcal{A}$ is \textit{operator biprojective} if $\mathcal{A}$ is
operator projective in $\mathcal{A}$-${\bf mod}$-$\mathcal{A}$.
\end{defn} 

Let $\mathcal{A}$ be a completely contractive Banach algebra. Then we denote the unitization $A \oplus \mathbb{C}$ of $A$ by $A_+$.
Let $X$ be a completely bounded left $A$-module with the multiplication map
	$$\pi_X : A \prt\, X\rightarrow X.$$
Then, $\pi_X$ can be extended to $\pi_{X,+} : A_+ \prt\, X \rightarrow X$ in a canonical way.
Using $\pi_{X,+}$ we have a useful characterization of operator projectivity by P. Woods (\cite[Corollary 3.20, Proposition 3.24]{W02}).

	\begin{prop}\label{prop-woods}
	Let $\mathcal{A}$ be a completely contractive Banach algebra and $X$ be a completely bounded left $A$-module.
	Then, $X$ is operator projective if and only if we have a completely bounded $A$-module map
		$$\rho : X \rightarrow A_+ \prt\, X,$$
	which is a right inverse of the extended multiplication map $\pi_{X,+}$.
	
	When $X$ is essential, $X$ is operator projective if and only if we have a completely bounded $A$-module map
		$$\rho : X \rightarrow A \prt\, X,$$
	which is a right inverse of the multiplication map $\pi_X$.
	\end{prop}

Let $A$ be a completely contractive Banach algebra, and let $X$ be a completely 
bounded left $A$-module. Let $\pi_A : A\widehat{\otimes}_A X \to X$ be the completely 
bounded left $A$-module mapping specified by 
$$\pi_A(a\otimes x)=ax \ \ (a\in A, x\in X).$$
In general, $\pi_A$ needs not be a complete isomorphism. The following proposition, which 
will play an important role in our analysis, 
gives us sufficient conditions for which $\pi_A$ is a complete isomorphism.

\begin{prop}
Let $A$ be a completely contractive Banach algebra with a b.a.i, and let $X$ be an essential completely 
bounded left $A$-module.
Then $\pi_A : A\widehat{\otimes}_A X \to X$ is a complete isomorphism. Moreover, $\pi_A$ is a complete 
isometry if $A$ has a contractive b.a.i.  
\end{prop}

\begin{proof}
Let $CB_A(A,X^*)$ denote the (Banach) space of the completely bounded right $A$-module morphisms.
$CB_A(A,X^*)$ inherits an operator space structure as a closed subspace of $CB(A,X^*)$. Now
consider the mapping $\Psi : X^* \to CB_A(A,X^*)$ defined by
	$$\Psi(x^*)(a)=x^*a \; (a\in A, x^*\in X^*).$$
It is clear that $\Psi=\pi_A^*$. Hence in order to show that $\pi_A$ is a 
complete isomorphism, it suffices to show that $\Psi$ is a complete isomorphism.
$\Psi$ is one-to-one since $\pi_A$ is onto by Cohen's factorization theorem 
\cite[Corollary 2.9.26]{D}. Moreover, since 
$A$ has a bounded approximate identity $\{e_\alpha\}_{\alpha \in I}$, then for every
$T\in CB_A(A,X^*)$, there is $x^*_T\in X^*$ such that $x^*_T$ is a $w^*$-cluster point of $\{ T(e_\alpha)\}$, 
and so,
$$T(a)=x_T^*a \;(a\in A).$$
That is $\Psi$ is onto. So it remains to show that both $\Psi$ and $\Psi^{-1}$ 
are completely bounded. 

Let $n\in \n$. For every $[x_{ij}^*]\in M_n(X^*)$, we have
\begin{eqnarray*} 
\|\Psi_n([x_{ij}^*])\|  &=& \sup \{\|[\Psi(x_{ij}^*)(a)]\| : \|a\|\leq 1 \} \\ 
& =& \sup \{\|[x_{ij}^*a]\| : \|a\|\leq 1 \}  \\ 
&= & \sup \{\|[x_{ij}^*]\otimes [a] \| : \|a\|\leq 1 \} \\ 
&\leq & \|[x_{ij}^*]\|.
\end{eqnarray*}    
Thus $\|\Psi_n \| \leq 1$. On the other hand, 
$$\sup \{\|[x_{ij}^*e_\alpha]\| : \alpha \in I \} \geq \|[x_{ij}^*]\|. \eqno{(1)}$$
To see this, first note that since $X$ is essential, for every $1\leq i,j \leq n$, 
$x_{ij}^*e_\alpha \stackrel{w^*}{\rightarrow}x_{ij}^*$ as $\alpha \to \infty$. Thus, 
for every $[x_{kl}]\in M_n(X)$, 
$$[\la x_{ij}^*e_\alpha \ , \ x_{kl}\ra ] \to [\la x_{ij}^*\ ,\ x_{kl} \ra]$$
in $M_{n^2}(\mathbb{C})$ as $\alpha \to \infty$. This implies that, for every $[x_{kl}]\in M_n(X)$
with $\|[x_{kl}]\| \leq 1$,
\begin{eqnarray*} 
\sup \{\|[x_{ij}^*e_\alpha]\| : \alpha \in I \} &\geq & 
\sup \{\|[\la x_{ij}^*e_\alpha \ ,\ x_{kl} \ra ]\| : \alpha \in I \}  \\ 
&\geq & \liminf_{\alpha \to \infty} \{\|[\la x_{ij}^*e_\alpha \ , \ x_{kl} \ra ]\| : \alpha \in I \}  \\ 
&= & \|[\la x_{ij}^* \ , \ x_{kl} \ra]\|. 
\end{eqnarray*}
Therefore  
\begin{eqnarray*} 
\sup \{\|[x_{ij}^*e_\alpha]\| : \alpha \in I \} &\geq & 
\sup \{\|[\la x_{ij}^* \ ,\ x_{kl} \ra ]\| : [x_{kl}]\in M_n(X), \|[x_{kl}]\| \leq 1 \}  \\  
&=& \|[x_{ij}^*]\|. 
\end{eqnarray*}
This proves (1). Now if we let $K=\sup \{\|e_\alpha \| : \alpha \in I \}$, it follows that
\begin{eqnarray*} 
K \|\Psi_n([x_{ij}^*])\|  &\geq & \sup \{\|\Psi(x_{ij}^*)(e_\alpha)]\| : \alpha \in I \} \\ 
&=&  \sup \{\|[x_{ij}^*e_\alpha]\| : \alpha \in I \} \\ 
&\geq&  \|[x_{ij}^*]\|.
\end{eqnarray*}
Thus $\|\Psi_n^{-1} \| \leq K$. Hence $\Psi$ is a complete isomorphisim.
Moreover, if $K=1$, then $\Psi$ is a complete isometry.
\end{proof}

By making use of the previous proposition and by appealing directly to \cite[Lemma 3.23]{W02}, we can immediately 
obtain the next theorem. 

	\begin{thm}\label{thm-essential}
	Let $A$ be a operator biprojective completely contractive Banach algebra with a b.a.i.,
	and let $X$ be an essential completely bounded left $A$-module.
	Then $X$ is operator projective in $A$-{\bf mod}.
	\end{thm}

It is important to note that the assumption that the completely contractive 
Banach algebra $\mathcal{A}$ has a b.a.i is crucial in the previous two statements. As we have 
previously observed, while 
$L^{1}(G)$ always has a bounded approximate identity, for $A(G)$ this is ture if and only if $G$ 
is amenable \cite{Lep}. Consequently, we will see that most of the positive results we 
obtain with respect to identifying operator projective $A(G)$-modules will require the assumption of 
amenability

Given a completely contractive Banach algebra $\mathcal{A}$, we let $\mathcal{A_{+}}$ denote the unitization 
of $\mathcal{A}$. The next proposition may be viewed as an operator analogue of Corollary 4.5 of \cite{Hel}.

	\begin{prop}\label{prop-proj}
	Let $A$ be a completely contractive Banach algebra, and let $X$ be a completely bounded left $A$-module.
	Suppose that $X$ is operator projective in $A$-{\bf mod} and $X$ or $A$ have OAP.
	Then for any non-zero element $x\in X$, there is a map $T \in {}_A CB(X, A_+)$ such that $T(x) \neq 0$.
	\end{prop}
\begin{proof}
Since $X$ is operator projective, we have a completely bounded left $A$-module map $\rho : X \rightarrow A_+ \prt X$
which is a right inverse of the module action $\pi_+ : A_+ \prt X \rightarrow X$ (Proposition \ref{prop-woods}).
Since $\pi_+ \circ \rho = id_X$ for any non-zero $x$, we have $\rho (x) \neq 0$.
We can suppose that $A$ has OAP since the proof for the case $X$ has OAP is the same.
It is clear that $A\oplus_\infty \mathbb{C}$ has OAP,
and so does $A_+ = A\oplus_1 \mathbb{C}$ since $A \oplus_1 \mathbb{C} \cong A \oplus_\infty \mathbb{C}$ completely isomorphically.
Then, the following canonical map is injective (Definition \ref{def-OAP}).
	$$A_+\prt X \rightarrow A_+ \otimes_{\min} X \hookrightarrow CB(A^*_+, X) \hookrightarrow CB(A^*_+, X^{**}) = (A^*_+ \prt X^*)^*$$
Thus, we can find $g\in A^*_+$ and $f\in X^*$ such that $(g\otimes f) (\rho (x)) \neq 0$ and consequently $(I_{A_+}\otimes f)(\rho (x)) \neq 0$.
If we set $T=(I_{A_+}\otimes f)\rho$, then $T$ is the map we desired.
\end{proof}

\section{The modules $A(G)$, $B(G)$, $A(G)^{**}$, and $UCB(\widehat{G})^*$}

In this section, we investigate the operator projectivity of $A(G)$, $B(G)$, $A(G)^{**}$, and 
$UCB(\widehat{G})^*$ in $A(G)$-{\bf mod}. We start by looking briefly at $A(G)$ itself to see 
when it is operator projective in $A(G)$-\textbf{ mod}. 

\subsection{Operator Projectivity of $A(G)$}

Dales and Polyakov have shown that $L^{1}(G)$ is always projective as a left Banach $L^{1}(G)$ module
 (see \cite[Theorem 2.4]{DP04}). In particular, $A(G)$ is operator projective in $A(G)$-{\bf mod} if $G$ is abelian. Moreover, 
This would suggest that we might expect that it is always true that $A(G)$ is operator projective in $A(G)$-{\bf mod}. 
In support of this assertion, in \cite{RX97} Ruan and Xu proved implicitly that $A(G)$ is operator right projective if $G$ is [IN]-group.
Indeed, they showed that (\cite[Lemma 3.2]{RX97}) $A(G)$ is operator right projective if 
there is a non-zero function $\xi \in L^2(G)$ such that
	$$\xi(tst^{-1})\Delta_G(t)^{-\frac{1}{2}} = \xi(t)$$
for any $s,t\in G$, where $\Delta_G$ is the modular function of $G$, which is true when $G$ is 
an [IN]-group. ($G\in [IN]$ if the identity element in $G$ has a compact neighbourdood that is 
invariant under inner automorphisms). 
Since $A(G)$ is symmetric module of itself we can easily transfer this result to the following.

	\begin{thm}
	If $G$ is [IN]-group, then $A(G)$ is operator projective in $A(G)$-{\bf mod}.
	\end{thm}

All compact groups and all abelian locally compact groups are $[IN]$-groups, as are all discrete groups. 

On the other hand, Aristov observed that there are connected groups whose Fourier algebras are not operator projective over itself.
In particular, $G = SL(3,\mathbb{R})$ is such an example (\cite{Ar05}). Unfortunately, at this point there is no obvious conjecture
as to when $A(G)$ is operator projective in $A(G)$-{\bf mod}.

\subsection{Operator Projectivity of $B(G)$}

It is shown in \cite[Theorem 2.6]{DP04} that $M(G)$ is a projective left $L^1(G)$-module
if and only if $G$ is discrete. Again, one hopes to be able to obtains the analogous result
for $B(G)$, i.e. $B(G)$ is an operator projective left $A(G)$-module if and only if
$G$ is compact. However there is an unexpected obstacle!

One key factor in determining the projectivity of $M(G)$ is the fact that, for non-discrete $G$, 
$\dim(M(G)/L^1(G))\geq 2$. As such, we need to establish that for non-compact $G$, $\dim(B(G)/A(G))\geq 2$. While 
one would certainly expect this to be true, we were unable to find this anywhere in the literature and 
surprisingly, we were not able to prove the statement in full generality. As such, we will not able to answer 
the question of when $B(G)$ is operator projective in full generality, although we have succeeded in showing 
that for most classes of non-compact groups, $B(G)$ is an operator projective left $A(G)$-module if and only if
$G$ is compact. Moreover, we can show that if $G$ is any group for which $B(G)$ is a left $A(G)$-module
then $G$ has an open compact subgroup.

\begin{defn}

Let $B_0(G)=\{u\in B(G) : u \textnormal{~ vanishes at ~} \infty\}$. 

Let $AP(G)$ denote the space of almost periodic functions on $G$. We will 
denote by $G^{ap}$ the almost periodic compactification of $G$. 

Let \[\widehat {G}_f=\{\pi \in \widehat {G} : \pi  \textnormal{~is finite dimensional~}\}.\]  

Next let  
\[\pi_f=\sum\limits_{\pi \in \widehat {G}_f} \oplus \pi .\]
 Then $A_{\pi _f}=B(G)\cap AP(G)$ and 
$A_{\pi _f}\cong A(G^{ap})$ \cite{BelFor}. 

We say that $G\in[AR]$ if the left regular representation $ \lambda $ on $G$ decomposes 
into a direct sum of irreducibles. 

We say that $G\in[AU]$ if the universal representation $\omega $ on $G$ decomposes 
into a direct sum of irreducibles. (See \cite{Taylor} for propeties of $[AR]$ and $[AU]$ groups.)
\end{defn}

\begin{lem}\label{L:co-dim-IN}
Let $G$ be a non-compact locally compact group. If $A(G)$ has finite codimension in $B(G)$, then 
$B_0(G)=A(G)$ and $B(G)=A_{\pi _f}\oplus A(G)$.
\end{lem}

\begin{proof}
Assume that $B_0(G)\not=A(G)$. Then it is clear that $A(G)$ also has finite
codimension in $B_0(G)$. But by \cite{Arsac} there exists 
a representation $\pi $ such that $B_0(G)=A_\pi(G)\oplus A(G)$. However, 
in this case, $\pi $ must be finite dimensional. Since coefficient 
functions of finite dimensional representations are almost periodic, and 
since $G$ is non-compact, 
this is impossible as the only almost periodic function that vanishes 
at $\infty$ is the constant function $0$. As such we have that $B_0(G)=A(G)$.

Since $A(G)$ has finite co-dimension in $B(G)$ we can again find a finite dimensional representation 
$\pi $ such that $B(G)=A_\pi \oplus A(G)$. But then $A_{\pi}\subseteq A_{\pi _f}$ so in fact we get that 
$A_{\pi}=A_{\pi _f}$ and we have that $B(G)=A_{\pi _f}\oplus A(G)$.
\end{proof} 

\begin{lem}\label{L:co-dim-sigma-compact}
Let $G$ be a non-compact locally compact group. If $A(G)$ has finite codimension in $B(G)$, then 
every open $\sigma $-compact subgroup is in $[AU]$. Moreover, $G$ has an open compact subgroup. 
\end{lem}

\begin{proof}
First let $H$ be an open $\sigma $-compact subgroup. Since $H$ is open it must be that $A(H)$ is also 
finite codimensional in $B(H)$. In particular, $B_0(H)=A(H)$. But then $H\in[AR]$ 
by \cite[Theorem 16]{BelFor}. But since $\pi _f$ also decomposes into a direct sum of irreducibles and 
since $\omega \cong  \pi_f \oplus \lambda$, we have that $\omega $ decoposes into a direct 
sum of irreducibles.  Hence $G\in[AU]$.

Let $H$ be an open almost connected subgroup. Then again, $A(H)$ is also 
finite codimensional in $B(H)$. But $G$ has a compact normal subgroup $K$ such that $H/K$ is separable.
As such $A(H/K)$ is separable. But once more $A(H/K)$ has finite codimension in $B(H/K)$. From this we 
can deduce that $B(H/K)$ is also separable. Because $B(H/K)$ is a dual space it has the 
Radon-Nikodym Property. This is equivalent to $H/K\in[AU]$. it follows from \cite[Theorem 4.11]{Taylor}, 
that $H/K$ is compact. This shows that $H$ is also compact. 
  
\end{proof}

\begin{thm}\label{T:finite-co-dim}
Let $G$ be a non-compact locally compact group. Then $A(G)$ does not have finite codimension in $B(G)$
in either of the following cases:\\
$($i$)$ $G\in [IN]$;\\
$($ii$)$ $G\in [MAP]$;\\
$($iii$)$ $G$ is connected;\\
$($iv$)$ $G$ is a Lie group.
\end{thm}

\begin{proof}
Suppose that $A(G)$ has finite codimension in $B(G)$. We show that in any of the cases
 (i)-(iv) that $A(G)$ has finite codimension in $B(G)$
implies that $G$ is compact, which is not possible.\\  
(i) By Lemma \ref{L:co-dim-IN}, we know that $B_0(G)=A(G)$. Thus $G$ must be compact 
by Corollary 17 in \cite{BelFor}.\\
(ii) Since $G$ is maximally almost periodic, $G$ injects into $G^{ap}$. If $G$ is infinite, then $G^{ap}$ 
is an infinite compact group. As such $A_{\pi _f}\cong A(G^{ap})$ is infinite dimensional. This means 
that we must have $A_{\pi _f}\subseteq A(G)$, and hence, $G$ must be compact by Lemma \ref{L:co-dim-IN}.\\ 
(iii) It follows directly from Lemma \ref{L:co-dim-sigma-compact} that $G$ is compact.\\
(iv) We know from Lemma \ref{L:co-dim-sigma-compact} that the connected componet of the identity, $G_e$, is compact. 
It follows that $A(G/G_e)$ is of finite codimension in $B(G/G_e)$. Since $G_e$ is open, $G/G_e$ is discrete. From (i), this means 
that $G/G_e$ must be finite, and hence that $G$ is compact.
\end{proof}

\begin{thm}
Let $G$ be a locally compact group. Then $B(G)$ is operator projective in 
$A(G)$-\m\ if and only if $G$ is compact or $\dim (B(G)/A(G))=1$. 
\end{thm}

\begin{proof}
``$\Leftarrow$" In this case, $B(G)$ is either $A(G)$ or the unitization of $A(G)$. 
Hence it is operator projective in $A(G)$-\m.\\
``$\Rightarrow$" Note that $B(G)$ is a completely contractive unital Banach algebra containing 
$A(G)$ as a closed ideal. Moreover $B(G)$ is faithful in $A(G)$-\m\ and, by hypothesis,
$\dim(B(G)/A(G))=1$. Thus, from the quantization of \cite[Proposition 1.3]{DP04}, $B(G)$ is not projective.
\end{proof}

\subsection{Operator Projectivity of $A(G)^{**}$}

We now turn our attention to the operator projectivity of $A(G)^{**}$ in $A(G)$-\m. 
In order to investigate this, we will need to make use of the so-called ``separation property" for Fourier algebras introduced 
by Kaniuth and Lau in \cite{KL}.
 
Let $G$ be a locally compact group, and let $H$ be a closed subgroup of $G$. We recall from \cite{KL} (see also
\cite{Der},\cite{DD}) that
$G$ has $H$-separation property if for every $x \notin H$, there is a continuous positive-definite function 
$f\in P(G)$ such that $f=1$ on $H$ and $f(x)\neq 1$. It follows from \cite{LL} that $G$ has $H$-separation
property if $H$ is either open, or compact, or normal in $G$.

\begin{lem}\label{L:sep property}
Let $G$ be a locally compact group, and let $H$ be a closed subgroup of $G$ such that
$G$ has $H$-seperation property. If $A(G)^{**}$ is operator projective in $A(G)$-\m, then
$A(H)^{**}$ is operator projective in $A(H)$-\m.
\end{lem}

\begin{proof}
Since $G$ has $H$-seperation property, by \cite[Proposition 3.1]{KL}, there exist a norm one 
projection $P : VN(G) \to VN(H)$ which is an $A(G)$-module homomorphism. Moreover, by 
\cite[Theorem 5.1.5]{Li}, $P$ is completely positive and thus completely bounded. Now suppose 
that $A(G)^{**}$ is operator projective in $A(G)$-\m. Let $R : A(G) \to A(H)$
be the restriction map, and let $\rho_G : A(G)^{**} \to A(G)_+ \widehat{\otimes} 
A(G)^{**}$ be a completely bounded left $A(G)$-module mapping which is a right inverse to the multiplication map 
$(\pi_G)_+ : A(G)_+ \widehat{\otimes} A(G)^{**} \to A(G)^{**}$. Hence we have the following commutating diagram:
\[
\xymatrix{
A(G)^{**} \ar@<.5ex>[rr]^{\rho_G} \ar@<.5ex>[d]^{R^{**}}
& & A(G)_+ \widehat{\otimes} A(G)^{**} \ar[d]^{R_+ \otimes R^{**}} \ar@<.5ex>[ll]^{(\pi_G)_+}  \\
A(H)^{**} \ar@<.5ex>[u]^{P^*}
& & A(H)_+ \widehat{\otimes} A(H)^{**} 
\ar[ll]^{(\pi_H)_+}  }
\]
Define $\rho_H:=(R_+ \otimes R^{**})\circ \rho_G \circ P^*$. It is clear that $\rho_H$
is a completely bounded left $A(H)$-module morphism. Moreover,

\begin{eqnarray*} 
(\pi_H)_+\circ \rho_H  &=& (\pi_H)_+ \circ (R_+ \otimes R^{**})\circ \rho_G \circ P^* \\ 
& =& R^{**}\circ (\pi_G)_+ \circ \rho_G \circ P^*  \\ 
&= & R^{**}\circ  P^* \\ 
&=& id_{A(H)^{**}}.
\end{eqnarray*}
Hence $A(H)^{**}$ is operator projective in $A(H)$-\m.
 
\end{proof}

\begin{prop}\label{P:second dual-amenable}
Let $G$ be an infinite amenable group. Then $A(G)^{**}$ is not operator projective in $A(G)$-\m.
\end{prop}

\begin{proof}
We prove the theorem's statement in several steps:

{\it Case I}: $G$ is compact. By a result of Zelmanov (\cite{Z}), $G$ contains an infinite, abelian compact 
subgroup $H$. Thus, from Lemma \ref{L:sep property}, $A(H)^{**}$ is projective in $A(H)$-\m\ which is impossible 
by \cite[Theorem 2.7]{DP04}.

{\it Case II}: $G$ is either connected or discrete. By the preceding part, we just need to 
consider the case where $G$ is not compact. Let $\{e_\alpha \}$ be a bounded approximate identity
in $A(G)$ bounded by 1. Let $E$ be the w$^*$-cluster point of $\{e_\alpha \}$ in $A(G)^{**}$, so 
that $E$ is a right identity in $A(G)^{**}$ and $\|E\|=1$. The algebra $A(G)^{**}$ is a completely
contractive symmetric Banach $B(G)$-module, and the mapping 
$$ f\mapsto f\cdot E \ , \ B(G) \to A(G)^{**}$$
is a complete isometry. Indeed, if we let 
$$I=\{ F\in A(G)^{**} : F=0 \ \text{on} \ C^*(G) \},$$
then $I$ is a closed two-sided (completely) complemented ideal in $A(G)^{**}$ and
$$A(G)^{**}\cong B(G)\oplus I.$$   
Put $J=\overline{A(G)I}$. We know that, from Cohen's factorization theorem 
\cite[Corollary 2.9.26]{D}, $J=A(G)I$. Clearly $A(G)A(G)^{**}\subseteq A(G)\oplus J$. 
Since from Theorem 
\ref{T:finite-co-dim}, $\dim (B(G)/A(G))$ is at least 2, there is $f\in B(G)$ such that 
$$\{f+A(G) , 1+A(G)\}$$ are linearly independent in $B(G)/A(G)$. Set $E_1=f\cdot E$. It follows that,
for every $g\in A(G)$,
$$ g\cdot E_1=(gf)\cdot E=gf.$$
Thus
$$ A(G)\cdot E_1 \subseteq A(G).$$
We show that $\{E+A(G)\oplus J , E_1+A(G)\oplus J\}$ are linearly independent. Suppose otherwise
so that there is a complex number $\beta\neq 0$ such that
$$(f-\beta 1_G)\cdot E \in A(G)\oplus J.$$ 
So $(f-\beta 1_G)\cdot E=h+F$, where $h\in A(G)$ and $F\in J$. Thus, for every $g\in A(G)$,
$$gh+g\cdot F=(f-\beta 1_G)\cdot (g\cdot E)=(f-\beta 1_G)g \in A(G).$$
This implies that $g\cdot F\in A(G)\cap J=\{0\}$. Therefore $F=0$ because
$J$ is an essential module on $A(G)$, and so, $F=\lim_{\alpha \to \infty} e_\alpha \cdot F=0$.
Hence $f-\beta 1_G=h \in A(G)$ which is impossible. It follows from the quantization of 
\cite[Corollary 1.4]{DP04} that $A(G)^{**}$ is not operator projective in $A(G)$-\m.

Now consider the general case. Let $G_e$ denote the connected component of identity $e$ in $G$.
Suppose that $A(G)^{**}$ is operator projective in $A(G)$-\m. Since $G_e$ is a closed normal subgroup 
of $G$, by Lemma \ref{L:sep property}, $A(G_e)^{**}$ is operator projective in $A(G_e)$-\m.
It follows from case II that $G_e=\{ e \}$ i.e. $G$ is totally disconnected. Let $H$ be a compact open
subgroup of $G$. Again by applying Lemma \ref{L:sep property}, we have that $A(H)^{**}$ is 
operator projective in $A(H)$-\m. Hence, from Case I, $H$ must be finite, and so, $G$ must be discrete
since $H$ is open in $G$. Therefore $G$ must be an infinite, discrete amenable group, But this 
contradicts Case II. Thus $A(G)^{**}$ is not operator projective in $A(G)$-\m. 

\end{proof}

\begin{thm}\label{T:proj-second dual}
Let $G$ be a locally compact group such that $A(G)^{**}$ is operator projective in $A(G)$-\m.
Then $G$ is discrete and contains no infinite amenable subgroup.
\end{thm}

\begin{proof}
First consider the case when $G$ is connected. Let $K$ be a compact normal subgroup of $G$ such that
$G/K$ is a Lie group. By Lemma \ref{L:sep property} and Proposition \ref{P:second dual-amenable},
$K$ is finite. Since $G/K$ is a Lie group, it has no ``small subgroup" i.e. there is an open
set $U$ of $G$ containing $K$ such that $K$ is the unique closed subgroup of $G/K$ contained in $U/K$. Let $V$ be an open
set contaning the identity $e$ such that 
$$V\cap K=\{e\} \ \text{and} \ \ VK\subseteq U.$$
This is possible since $K$ is finite. We claim that $\{ e\}$ is the unique closed subgroup of $G$ contained in $V$.
To see this, let $H$ be a closed subgroup of $G$ contained in $V$. Then $HK\subseteq U$ and 
$$ HK/K\cong H/H\cap K=H.$$
Therefore $H=\{e\}$. This implies that $G$ is a connected Lie group.

Let $R$ be the radical of $G$ i.e. the largest connected, solvable (closed) subgroup of 
$G$. Since $R$ is amenable and normal, it follows from Lemma \ref{L:sep property} and Proposition
\ref{P:second dual-amenable} that $R=\{e\}$. Hence $G$ is a connected semisimple Lie group.
It follows from the Iwasawa's decomposition theorem that if $G$ is non-trival, it
contains an infinite compact subgroup. However this is not possible because of 
Lemma \ref{L:sep property} and Proposition \ref{P:second dual-amenable}. Thus $G$ is a trivial group.   

Now consider the general case. Let $G_e$ be the connected component of the identity $\{e\}$ in $G$.
By Lemma \ref{L:sep property} and Proposition \ref{P:second dual-amenable}, $A(G_e)^{**}$ is operator projective in $A(G_e)$-\m. Hence, from the preceding part, $G_e=\{e\}$ i.e. $G$ is totally disconnected. A similar argument
to the one made in the proof of Proposition \ref{P:second dual-amenable} shows that $G$ must be discrete.
Finally since every subgroup of a discrete group is open, it follows again from Lemma \ref{L:sep property} and Proposition \ref{P:second dual-amenable} that $G$ contains no infinite amenable subgroup.
\end{proof}

\subsection{Operator Projectivity of $UCB(\widehat{G})^*$}

Finally we investigate the operator projectivity of $UCB(\widehat{G})^*$ in $A(G)$-\m.
The results in this section would be analogous to those of the preceding section. 
Indeed, our main theorem is the following which is similar to Theorem \ref{T:proj-second dual}:

\begin{thm}\label{T:proj-UCB dual}
Let $G$ be a locally compact group such that $UCB(\widehat{G})^*$ is operator projective in $A(G)$-\m.
Then $G$ is discrete and contains no infinite amenable subgroup.
\end{thm}

\begin{proof}
Let $H$ be a closed subgroup of $G$ such that $G$ has $H$-separation property, and let $P$ and $R$ be 
the mapping considered in the proof of Lemma \ref{L:sep property}. It is routine to verify that
$P$ maps $UCB(\widehat{G})$ onto $UCB(\widehat{H})$ and $R^*$ maps $UCB(\widehat{H})$ into
$UCB(\widehat{G})$. Thus the analogous of Lemma \ref{L:sep property} holds in the case of $UCB(\widehat{G})^*$
if we modify its proof.

Similarly, Proposition \ref{P:second dual-amenable} also holds if we replace $A(G)^{**}$ with $UCB(\widehat{G})^*$.
Indeed, the proof will go through exactly the same considering the fact that, for amenable $G$, 
$$I=\{ F\in UCB(\widehat{G})^* : F=0 \ \text{on} \ C^*(G) \},$$
is a closed two-sided (completely) complemented ideal in $UCB(\widehat{G})^*$ and
$$UCB(\widehat{G})^* \cong B(G)\oplus I.$$  
The final result follows if we modify the proof of Theorem \ref{T:proj-second dual} for
$UCB(\widehat{G})^*$ instead of $A(G)^{**}$.
\end{proof}

\section{The modules $C^*_r(G)$, $UCB(\widehat{G})$ and $VN(G)$}\label{C*r(G)VN(G)}

Since $A(G)^* = VN(G)$, we have a dual $A(G)$-module structure on $VN(G)$ given by
	$$\left\langle f \cdot a, g \right\rangle := \left\langle a, g f\right\rangle$$
for any $f,g \in A(G)$ and $a\in VN(G)$. In particular, for $a = L_h$ with $h\in L^1(G)$ we have
	$$f \cdot L_h = L_{f h},$$
where $L_h$ is the left convolution with respect to $h$ on $L^2(G)$.
Thus, this gives us an $A(G)$-module structure on $C^*_r(G)$.
Moreover, since $UCB(\widehat{G}) = \overline{A(G)\cdot VN(G)}$, we get an $A(G)$-module structure given by
	$$f\cdot(g\cdot a) = (f g)\cdot a$$
for any $f,g\in A(G)$ and $a\in VN(G)$.
Clearly, $C^*_r(G)$ and $UCB(\widehat{G})$ are essential as $A(G)$-modules.

First, we consider the cases where projectivity holds.
It is well known that $A(G)$ is an operator biprojective completely contractive Banach algebra when $G$ is discrete (\cite{W02})
and $A(G)$ has a b.a.i. when $G$ is amenable \cite{Lep}.
Thus, we can apply Theorem \ref{thm-essential} to get the following result.
	\begin{thm}\label{thm-pos-Cr(G)}
	Let $G$ be a discrete and amenable group. Then $C^*_r(G)$ and $UCB(\widehat{G})$ are operator projective in $A(G)$-{\bf mod}.
	\end{thm}

Now we focus on the cases where projectivity fails. The following information about the module mapping spaces is the key step to the proof. 
	\begin{prop}\label{prop-trivial-mapping-space}
	Let $G$ be a non-discrete group. Then
		$${}_{A(G)}B(C^*_r(G), A(G)) = 0.$$
	\end{prop}
\begin{proof}
We pick any $T \in {}_{A(G)}B(C^*_r(G) , A(G))$, and we claim that $g := T(L_{1_K}) = 0$
for any fixed open subset $K$ of $G$ with compact closure and with positive measure.
First, since $T$ is a left module map, we can observe that
	\begin{equation}\label{eq1}
	T(L_{f 1_K}) = T(f \cdot L_{1_K}) = f g
	\end{equation}
for any $f \in A(G)\cap L^1(G)$.
Actually, we can extend the identity in \eqref{eq1} for any $f \in L^1(G)$ since $A(G)\cap L^1(G)$ is dense in $L^1(G)$.
Indeed, we choose a sequence of functions $(f_i)_{i\ge 1} \subseteq A(G) \cap L^1(G)$ converging to $f$ in $L^1(G)$.
Then we have
	$$L_{f_i 1_K} \stackrel{C^*_r(G)}{\longrightarrow} L_{f 1_K}$$
and consequently
	$$f_i g = T(L_{f_i 1_K}) \stackrel{A(G)}{\longrightarrow} T(L_{f 1_K}).$$
By passing to an appropriate subsequence we can assume that $f_i$ is converging to $f$ almost everywhere.
Thus, we can conclude that
	\begin{equation}\label{eq2}
	T(L_{f 1_K}) = f g\;\text{a.e.}
	\end{equation}
for any $f\in L^1(G)$.

Now we consider the following composition of bounded maps:
	$$\Phi : L^1(G) \stackrel{\lambda}{\longrightarrow} C^*_r(G) \stackrel{T}{\longrightarrow} A(G) \stackrel{j}{\hookrightarrow} L^{\infty}(G),$$
where $\lambda$ is the left regular representation on $L^1(G)$ and $j$ is the canonical embedding. Then, we have
	$$\Phi(f) = f g$$
for any $f\in C_c(G)$ with $\text{supp}(f) \subseteq K$, which implies that
	$$g=0 \;\text{on $K$}.$$
Indeed, if $g(x) \neq 0$ at some point $x\in K$,
then we can always find $f$ which has norm 1 in $L^1(G)$ and $\text{supp}(f) \subseteq K$ but $f(x)$ is arbitrarily large.
On the other hand, by \eqref{eq2} with $f = 1_K$ we have
	$$g = T(L_{1_K}) = g 1_K,$$
so that $g = 0$ on $G\backslash K$. Consequently $T(L_{1_K}) = g \equiv 0$.

Now we pick any non-zero $f \in C_c(G)$ and set $K = \{ x\in G : f(x)\neq 0\}$.
Then since $G$ is not discrete, $K$ is an open set with compact closure and positive measure.
Thus, from the above calculation and \eqref{eq2} we have
	$$T(L_f) = T(L_{f 1_K}) = f g = 0.$$
By a standard density argument we have $T = 0$, so that
	$${}_{A(G)}B(C^*_r(G), A(G)) = 0.$$
\end{proof}

We also need the following transference result.
	\begin{lem}\label{lem-opensubgp}
	Let $H$ be an open subgroup of a locally compact group $G$.
	If $C^*_r(G)$ (resp. $UCB(\widehat{G})$, $VN(G)$) is operator projective in $A(G)$-{\bf mod},
	then $C^*_r(H)$ (resp. $UCB(\widehat{H})$, $VN(H)$) is operator projective in $A(H)$-{\bf mod}.
	\end{lem}
\begin{proof}
Since $H$ is an open subgroup we have a completely isometric embedding \cite{ForWood}
	$$j_1 : A(H) \hookrightarrow A(G),\, f \mapsto \widetilde{f},$$
where $\widetilde{f}$ is the extension of $f$ to $G$ by assigning 0 outside of $H$,
and the following restriction map is also a complete contraction \cite{ForWood}.
	$$R_1 : A(G) \rightarrow A(H),\, g \mapsto g|_H.$$
Moreover, we have a completely contractive projection $j^*_1$ and a completely contractive map $R^*_1$ with their restrictions
	$$R_\infty = j^*_1|_{C^*_r(G)} : C^*_r(G) \rightarrow C^*_r(H),\, L_f \mapsto L_{f|_H}$$
and
	$$j_\infty = R^*_1|_{C^*_r(H)} : C^*_r(H) \rightarrow C^*_r(G),\, L_f \mapsto L_{\widetilde{f}}.$$
Now we suppose that $C^*_r(G)$ is operator projective in $A(G)$-{\bf mod}.
Then there is a completely bounded left $A(G)$-module map $\rho_G : C^*_r(G) \rightarrow A(G) \prt C^*_r(G)$,
which is a right inverse of the multiplication map $\pi_G : A(G) \prt C^*_r(G) \rightarrow  C^*_r(G)$.
Moreover, we have the following commutative diagram.

	\begin{equation}\label{diagram1}
	\xymatrix{
	C^*_r(G) \ar@<.5ex>[rr]^{\rho_G} \ar@<.5ex>[d]^{R_\infty}
	& & A(G)\prt C^*_r(G) \ar[d]^{R_1 \otimes R_\infty} \ar@<.5ex>[ll]^{\pi_G}  \\
	C^*_r(H) \ar@<.5ex>[u]^{j_\infty}
	& & A(H)\prt C^*_r(H)
	\ar[ll]^{\pi_H}}
	\end{equation}

If we set $\rho_H = (R_1\otimes R_\infty)\circ \rho_G \circ j_\infty$,
then it is straightforward to check that $\rho_H$ is a completely bounded left $A(H)$-module map, which is a right inverse of $\pi_H$.
The proofs for $UCB(\widehat{G})$ and $VN(G)$ are similar.
\end{proof}

	\begin{thm}\label{thm-notproj-Cr(G)}
	Let $G$ be a non-discrete locally compact group.
	Then $C^*_r(G)$, $UCB(\widehat{G})$, and $VN(G)$ are not operator projective in $A(G)$-{\bf mod}.
	\end{thm}
\begin{proof}
We first consider $C^*_r(G)$. From the structure theory (\cite[Proposition 12.2.2]{Pa01})
we know that there is an almost connected open subgroup $H$ of $G$.
Since $G$ is non-discrete, $H$ is also non-discrete, and since $VN(H)$ is injective (\cite{Pat}), $A(H)$ has OAP.
Now we suppose $C^*_r(G)$ is operator projective in $A(G)$-{\bf mod},
then by Lemma \ref{lem-opensubgp}, $C^*_r(H)$ is operator projective in $A(H)$-{\bf mod}.
By Proposition \ref{prop-proj} for any non-zero element $x\in C^*_r(H)$ we can find a map $T \in {}_{A(H)} CB(C^*_r(H), A(H)_+)$ such that $T(x)\neq 0$.
By multiplying an appropriate non-zero function in $A(H)$
we can actually find $T' \in {}_{A(H)} CB(C^*_r(H), A(H))$ such that $T'(x)\neq 0$.
However, this is impossible by Proposition \ref{prop-trivial-mapping-space}.

The case of $VN(G)$ is similar. Let $H$ be again an almost connected open subgroup of $G$.
We fix an open subset $K$ of $H$ with compact closure and with positive measure.
Then, it can be shown that $T(L_{1_K}) = 0$ for any $T \in {}_{A(H)} B(VN(H), A(H))$ as in the proof of Proposition \ref{prop-trivial-mapping-space},
which implies $VN(G)$ is not operator projective by Lemma \ref{lem-opensubgp} and Proposition \ref{prop-proj}.
For the case $UCB(\widehat{G})$, we choose $H$ and $K$ as above, and we choose a nonzero $h\in A(H)$ supported in $K$.
Then, it can be shown that $T(h\cdot L_{1_K}) = 0$ for any $T \in {}_{A(H)} B(UCB(\widehat{H}), A(H))$,
which implies $UCB(\widehat{G})$ is not operator projective.

\end{proof}

Since $VN(G)$ is not essential we can not use the same approach for the positive answer.
Indeed, we have a more restricted result as you can see in Theorem \ref{thm-VN(G)} below.

	\begin{lem}\label{lem-proj}
	Let $G$ be an amenable discrete group. Suppose that $VN(G)$ is operator projective in $A(G)$-{\bf mod}.
	Then there is a bounded projection from $VN(G)$ onto $C^*_r(G)$.
	\end{lem}  
\begin{proof}
When $G$ is discrete it is clear that $A(G) \cdot VN(G) \subseteq C^*_r(G)$. Since $G$ is amenable we have a b.a.i. in $A(G)$.
Then, the rest of the proof is the same as \cite[Lemma 3.2]{DP04}.
\end{proof}

The next two results are of independent interest.

The first proposition and its proof was suggested to us by Bill Johnson and Roger Smith. (We also want to thank Roger Smith and 
Gilles Pisier for bring to our attention the manuscript \cite{ABU} where a more general version of 
the proposition below is establised.) 

\begin{prop}\label{P:com-A-VN(A)}
Let $H$ be a Hilbert space and $A \subseteq B(H)$ be an infinite separable $C^*$-algebra, and let $VN(A)$ be the
von Neumann algebra generated by $A$. Then $A$ is not complemented in $VN(A)$.
\end{prop}

\begin{proof}
We first claim that $A$ contains a copy of $c_0$. Look at a masa $B$ in $A$.
If it is finite dimensional then it follows easily that $A$ is finite dimensional.
Thus $B$ is infinite dimensional, and so, it is $*$-isomorphic to $C(X)$ for an infinite metric space $X$.
Choose infinitely many disjoint open balls in $X$, and choose a sequence of positive functions of norm 1 supported in these balls,
say $(h_i)_{i\ge 1}$. Then $(c_i) \mapsto \sum c_ih_i$ embedds $c_0$ into $A$ and also $\ell^{\infty}$ into $VN(A)$.

Now suppose there is a projection from $VN(A)$ onto $A$. Since $c_0$ is separably injective (\cite{So62}), there is a projection from $A$ onto $c_0$.
Combining with the projection from $VN(A)$ onto $A$ and restricting this to $\ell^{\infty}$
gives a bounded projection from $\ell^{\infty}$ onto $c_0$ which contradicts the Phillips Lemma. Thus, $A$ is not complemented in $VN(A)$.
\end{proof}

\begin{thm}\label{thm-non-complemented}
Let $G$ be an infinite locally compact group. Then $C^*_r(G)$ is not complemented in $VN(G)$.
\end{thm}

\begin{proof}
Suppose that there is a bounded projection $P$ from $VN(G)$ onto $C^*_r(G)$.
Let $H$ be an open almost connected subgroup of $G$. It is clear that
$P_H:=1_HP$ is a bounded projection from $VN(H)$ onto $C^*_r(H)$. Now let $K$
be a compact normal subgroup of $H$ such that $H/K$ is a Lie group, and so, in particular,
$C_r^*(H/K)$ is separable. Let $\Gamma : C_r^*(H) \to C_r^*(H/K)$ be the averaging operator over
the Haar measure of $K$ i.e. 
	$$\Gamma(T)=\int_K T*\delta_k dk \ \ (T\in C_r^*(G)).$$
It is straightforward to verify that the restriction of $\Gamma\circ P_H$ on
$VN(H/K)$ is a bounded projection from $VN(H/K)$ onto $C_r^*(H/K)$. However this is
impossible from Proposition \ref{P:com-A-VN(A)}. Thus $C_r^*(G)$ is not complemented in $VN(G)$. 
\end{proof}

If we combine Lemma \ref{lem-proj} with Lemma \ref{lem-opensubgp} and Theorem \ref{thm-non-complemented}, then we get the following.

	\begin{thm}\label{thm-VN(G)}
	Let $G$ be a discrete group containing an infinite amenable subgroup.
	Then $VN(G)$ is not operator projective in $A(G)$-{\bf mod}.
	\end{thm}

Now we consider the case of discrete groups containing $\mathbb{F}_2$, the free group with 2 generators,
equivalently, containing $\mathbb{F}_\infty$, the free group with infinitely many generators.
Actually, we have negative results for those groups, which shows that amenability does matter.
The proof depends on the operator space structure of the span of free generators in $C^*_r(\mathbb{F}_\infty)$ and $A(\mathbb{F}_\infty)$.

Recall that $R_n \cap C_n$ (or $RC^\infty_n$) and $R_n + C_n$ (or $RC^1_n$) are the intersection and the sum of $R_n$ and $C_n$,
$n$-dimensional row and column Hilbert spaces, respectively (\cite{P03}).
Note that operator space structures of $R_n$ and $C_n$ are determined as follows
	$$\norm{\sum^n_{i=1}A_i \otimes e_i}_{M_m(R_n)} = \norm{\sum^n_{i=1}A_iA^*_i}_{M_m}^{\frac{1}{2}}$$
and
	$$\norm{\sum^n_{i=1}A_i \otimes e_i}_{M_m(C_n)} = \norm{\sum^n_{i=1}A^*_iA_i}_{M_m}^{\frac{1}{2}}$$
for any $(A_i)^n_{i=1} \subseteq M_m$, $m\in \n$, where $(e_i)^n_{i=1}$ is the standard basis of $R_n$ and $C_n$.
Moreover, for any $(A_i)^n_{i=1} \subseteq M_m$, $m\in \n$, we have
	$$\norm{\sum^n_{i=1}A_i \otimes e_i}_{M_m(R_n\cap\, C_n)}
	= \max \Bigg\{ \norm{\sum^n_{i=1}A_iA^*_i}_{M_m}^{\frac{1}{2}}, \norm{\sum^n_{i=1}A^*_iA_i}_{M_m}^{\frac{1}{2}} \Bigg\}$$
and 
	$$\norm{\sum^n_{i=1}A_i \otimes e_i}_{M_m(R_n +C_n)}
	= \inf_{A_i = B_i + C_i} \Bigg\{ \norm{\sum^n_{i=1}B_iB^*_i}_{M_m}^{\frac{1}{2}}+ \norm{\sum^n_{i=1}C^*_iC_i}_{M_m}^{\frac{1}{2}} \Bigg\},$$
where the infimum runs over all possible $A_i = B_i + C_i$ and $(e_i)^n_{i=1}$ is the standard basis of $R_n \cap C_n$ and $R_n +C_n$.
Note also that $(R_n\cap C_n)^* = R_n +C_n$.

	\begin{thm}\label{thm-free-Cr(G)}
	Let $G$ be a discrete group containing $\mathbb{F}_2$.
	Then, $C^*_r(G)$ (resp. $UCB(\widehat{G})$, $VN(G)$) is not operator projective in $A(G)$-{\bf mod}.
	\end{thm}
\begin{proof}
Note that it is enough to show the theorem for the case $G = \mathbb{F}_\infty$
because of Lemma \ref{lem-opensubgp} and the fact that $\mathbb{F}_2$ contains $\mathbb{F}_\infty$.
We first consider the case $C^*_r(G)$ and suppose that $C^*_r(G)$ is operator projective in $A(G)$-{\bf mod}.
Since $C^*_r(G)$ is essential, we have a completely bounded left $A(G)$-module map
	$$\rho : C^*_r(G) \rightarrow A(G) \prt C^*_r(G),$$
which is a right inverse of the multiplication map $\pi$. Now we claim that $\rho$ is of the form
	\begin{equation}\label{right-inverse}
	\rho(\lambda(s)) = \delta_s \otimes (\lambda(s) + x_s)
	\end{equation}
for any $s\in G$, where $\lambda$ is the left regular representation of $G$ and $x_s \in X$ with $C^*_r(G) = X\oplus \mathbb{C}\lambda(s)$.

Indeed, for any fixed $s\in G$ we have $(x_i) \subseteq A(G)\otimes C^*_r(G)$ (algebraic tensor product) such that
	$$x_i = \sum^{n_i}_{j=1} a_j \otimes b_j\;\, \text{and}\;\, \rho(\lambda(s)) = \lim_i x_i.$$
Then we have $\rho(\lambda(s)) = \rho(\delta_s\cdot \lambda(s)) = \delta_s\cdot \rho(\lambda(s)) = \lim_i \delta_s\cdot x_i$ and
	$$\delta_s \cdot x_i = \sum^{n_i}_{j=1}\delta_s\cdot a_j \otimes b_j = \sum^{n_i}_{j=1}a_j(s)\delta_s \otimes b_j
	= \delta_s\otimes \Big(\sum^{n_i}_{j=1} a_j(s) b_j\Big) \in \delta_s \otimes C^*_r(G).$$
Since $\delta_s \otimes C^*_r(G)$ is a norm closed subspace of $A(G) \prt C^*_r(G)$ we have
	$$\rho(\lambda(s)) \in \delta_s \otimes C^*_r(G),$$
which implies that $\rho(\lambda(s)) = \delta_s \otimes (C \cdot \lambda(s) + x_s)$
for some constant $C$ and $x_s \in X$. Since $\pi \circ \rho (\lambda(s)) = \lambda(s)$ we have $C=1$.

Now we set
	$$E^\infty_n = \text{span}\{\lambda(g_i)\}^n_{i=1} \subseteq C^*_r(G), \; E^1_n = \text{span}\{\delta_{g_i}\}^n_{i=1} \subseteq A(G),$$
	$$P^\infty_n : C^*_r(G) \rightarrow E^\infty_n,\; \lambda(g_i) \mapsto
	\left\{ \begin{array}{ll} \lambda(g_i) & 1\le i\le n\\ 0 & \text{otherwise} \end{array}\right.,$$
and
	\begin{equation}\label{P1}
	P^1_n : A(G) \rightarrow E^1_n,\; \delta_{g_i} \mapsto
	\left\{ \begin{array}{ll} \delta_{g_i} & 1\le i\le n\\ 0 & \text{otherwise} \end{array}\right.,
	\end{equation}
where $(g_i)_{i \ge 1}$ is the set of free generators in $G=\mathbb{F}_\infty$.
It is well known (\cite[section 9.7]{P03}) that
	$$\phi : E^\infty_n \rightarrow R_n \cap C_n,\; \lambda(g_i) \mapsto e_i$$
and
	\begin{equation}\label{E1}
	\psi : E^1_n \rightarrow R_n + C_n,\; \delta_{g_i} \mapsto e_i
	\end{equation}
are complete 2-isomorphisms, and it is also well known (\cite[section 9.7]{P03})
that $P^\infty_n$ and $P^1_n$ are completely bounded with cb-norms $\le 2$.
By composing the above maps with $\rho$ we get the following completely bounded map with cb-norm independent of $n$: 
	$$\Phi : R_n \cap C_n \stackrel{\phi^{-1}}{\longrightarrow} C^*_r(G) \stackrel{\rho}{\longrightarrow} A(G) \prt C^*_r(G) \stackrel{P^1_n\otimes
	P^\infty_n}{\longrightarrow} E^1_n \prt E^\infty_n \stackrel{\psi \otimes \phi}{\longrightarrow} (R_n + C_n) \prt (R_n\cap C_n),$$
where
	$$\Phi(e_i) = e_i \otimes \Big(e_i + \sum_{1\le j\neq i\le n}\alpha_{i,j}e_j \Big)$$ for some constants $(\alpha_{i,j})_{1\le j\neq i\le n}$.
Note that
	$$(R_n + C_n) \prt (R_n\cap C_n) \cong CB(R_n\cap C_n)^*$$
completely isometrically under the duality bracket
	$$\la x\otimes y, T \ra = \la T(y), x \ra$$
for any $x\in R_n+C_n$, $y\in R_n\cap C_n$, and $T \in CB(R_n\cap C_n)$.

Thus we have
	\begin{align*}
	\norm{\Phi \Big(\sum^n_{i=1}e_i\Big)}_{CB(R_n\cap C_n)^*} &
	\ge \frac{\abs{\left\langle \Phi\Big(\sum^n_{i=1}e_i\Big), id_{R_n\cap\, C_n} \right\rangle}}{\norm{id_{R_n\cap\, C_n}}_{cb}}\\
	& = \abs{ \sum^n_{i=1} \la e_i + \sum_{1\le j\neq i\le n}\alpha_{i,j}e_j, e_i \ra }\\
	& = \abs{ \sum^n_{i=1} \la e_i, e_i \ra } = n.
	\end{align*}
On the other hand, since $R_n \cap C_n$ is just $\ell^2_n$ as a Banach space, we have
	$$\norm{\Phi \Big(\sum^n_{i=1}e_i\Big)}_{CB(R_n\cap C_n)^*} \le \norm{\Phi} \norm{\sum^n_{i=1}e_i}_{R_n \cap C_n} = \norm{\Phi}\sqrt{n},$$
which is a contradiction when $n$ is large enough.

The proof for $UCB(\widehat{G})$ and $VN(G)$ are similar, but we need to be careful when showing \eqref{right-inverse} for $VN(G)$
since $VN(G)$ is not essential.
Indeed, for any fixed $s\in G$ we have $\rho(\lambda(s)) = \lim_i x_i$ for some
$x_i = \sum^{n_i}_{j=1}(a_j \oplus c_j) \otimes b_j \in A(G)_+ \otimes VN(G)$, where $a_j \in A(G)$, $b_j\in VN(G)$, and $c_j$'s are scalars.
Then, we have
	\begin{align*}
	\delta_s \cdot x_i & = \sum^{n_i}_{j=1}\delta_s \cdot (a_j \oplus c_j) \otimes b_j = \sum^{n_i}_{j=1}(a_j(s) + c_j) \delta_s \otimes b_j\\
	& = \delta_s \otimes \Big(\sum^{n_i}_{j=1}(a_j(s) + c_j) b_j\Big) \in \delta_s \otimes VN(G).
	\end{align*}
The rest of the proof is the same.
\end{proof}

\section{The modules $L^p(VN(G))$ for $1<p<\infty$}

In this section we will assume that the reader is familiar with standard materials about
complex interpolation of Banach spaces (\cite{BL76}) and operator spaces (\cite{P96}).

\subsection{Noncommutative $L^p$ and complex interpolation}
In this subsection we will briefly describe complex interpolation approach to noncommutative $L^p$ by Izumi (\cite{Iz97}),
which is a generalization of the results by Kosaki (\cite{Ko84a}) and Terp (\cite{Te81}).

Let $\M$ be a von Neumann algebra with a normal semifinite faithful (shortly n.s.f.) weight $\varphi$.
Let $n_\varphi = \{x\in \M : \varphi(x^*x) <\infty\}$ and $\Lambda$ be the canonical embedding of $n_\varphi$ into $\Hi_\varphi$,
where $\Hi_\varphi$ is the Hilbert space obtained by the completion of $n_\varphi$ with the inner product
$\left\langle x, y\right\rangle_\varphi = \varphi(y^* x)$ for any $x, y \in n_\varphi$.
We consider the closed anti-linear map $S$ densely defined on $\Hi_\varphi$ by
	$$S(\Lambda x) = \Lambda (x^*)\; \text{for any}\; x\in n^*_\varphi \cap n_\varphi.$$
Let $J$ and $\Delta$ be the modular conjugation and the modular operator, respectively, obtained by the polar decomposition
	$$S = J\Delta^{\frac{1}{2}}.$$
Now we consider the full Tomita algebra $\mathfrak{a}_0$ defined by
	$$\mathfrak{a}_0 = \Lambda^{-1}(\mathfrak{A}_0),\; \text{where}$$
	$$\mathfrak{A}_0 = \{\xi \in \cap^\infty_{n=-\infty} D(\Delta^n) : \Delta \xi \in \Lambda(n^*_\varphi \cap n_\varphi), n\in \n\}.$$
Note that $\mathfrak{A}_0$ is dense in $\Hi_\varphi$.

The modular operator $\Delta$ give rise to an one-parameter modular automorphism group $\{\sigma^\varphi_t\}_{t\in\mathbb{R}}$ on $\M$ given by
	$$\sigma^\varphi_t(x) = \Delta^{it}x\Delta^{-it}\; \text{for any}\; t\in \mathbb{R}.$$
It is well known (\cite{Ta03}) that 
	\begin{equation}\label{analytic-elts}
	\text{{\it$\{\sigma^\varphi_t\}_{t\in\mathbb{R}}$ extends to a complex one-parameter group
	$\{\sigma^\varphi_\alpha\}_{\alpha \in \mathbb{C}}$ on $\mathfrak{a}_0$.}}
	\end{equation}

Now we introduce an important class in complex interpolation of noncommutative $L^p$ spaces. For any $\alpha \in \mathbb{C}$ we define
	\begin{equation}\label{def-L-alpha}
	L_{(\alpha)} = \left\{ x\in \M:  \begin{array}{ll} \text{there is a unique functional $\varphi^{(\alpha)}_x \in \M_*$ such that}\\
	\text{$\varphi^{(\alpha)}_x(y^* z) = \left\langle xJ\Delta^{\overline{\alpha}} \Lambda (y), J\Delta^{-\alpha} \Lambda(z) \right\rangle$
	for any $y,z\in \mathfrak{a}_0$}  \end{array}\right\}
	\end{equation}
with the norm $\norm{x}_{L_{(\alpha)}} = \max\{ \norm{x}_{\M}, \norm{\varphi^{(\alpha)}_x}_{\M_*} \}$.
Let $\mathfrak{a}^2_0$ be the algebraic linear span of the elements of the form $y^*z$ for $y,z \in \mathfrak{a}_0$.
Then, it is known (\cite{Iz97}) that
	$$\mathfrak{a}^2_0 \subseteq L_{(\alpha)}\; \text{for any}\; \alpha \in \mathbb{C}.$$
Note again that $\mathfrak{a}^2_0$ contains enough elements, so that $\Lambda(\mathfrak{a}^2_0)$ is dense in $\Hi_\varphi$.
For $\alpha \in \mathbb{C}$ we consider
	$$i^\infty_{(\alpha)} : L_{(\alpha)} \rightarrow \M,\; \text{the natural inclusion}$$
and
	$$i^1_{(\alpha)} : L_{(\alpha)} \rightarrow \M_*,\; x\mapsto \varphi^{(\alpha)}_x.$$ 
Then, the following diagram
	\begin{equation}\label{diagram-interpol}
	\xymatrix{&\M \ar[dr]^{(i^1_{(-\alpha)})^*} & \\L_{(\alpha)} \ar[ur]^{i^\infty_{(\alpha)}} \ar[dr]_{i^1_{(\alpha)}} && L^*_{(-\alpha)}\\
	&\M_* \ar[ur]_{(i^\infty_{(-\alpha)})^*}&}
	\end{equation}
is commutative (\cite[Theorem 2.5]{Iz97}), that is, we have
	$$\varphi^{(\alpha)}_x (y) = \varphi^{(-\alpha)}_y(x)\; \text{for any}\; x\in L_{(\alpha)}, y\in L_{(-\alpha)}.$$
Using $(i^1_{(-\alpha)})^*$ and $(i^\infty_{(-\alpha)})^*$ as identification maps we define
	$$L^p_{(\alpha)}(\M, \varphi) = [\M, \M_*]_{\frac{1}{p}}\; \text{for}\; 1<p<\infty.$$
It is proved in \cite{Iz97} that the definition of $L^p_{(\alpha)}(\M, \varphi)$ is independent of the choice of $\alpha$
up to an isometric isomorphism and is isometric to the $L^p$ space in the sense of Haagerup. Thus, we will simply write $L^p(\M)$.

If we interpolate the maps $i^1_{(\alpha)}$ and $i^\infty_{(\alpha)}$ and the maps $(i^1_{(-\alpha)})^*$ and $(i^1_{(-\alpha)})^*$ in \eqref{diagram-interpol}, respectively, then we get contractions $i^p_{(\alpha)}$ and $(i^p_{(-\alpha)})^*$ and we get another commutative diagram
	$$\xymatrix{&\M \ar[drr]^{(i^1_{(-\alpha)})^*} & \\
	L_{(\alpha)} \ar[ur]^{i^\infty_{(\alpha)}} \ar[dr]_{i^1_{(\alpha)}} \ar[r]^{i^p_{(\alpha)}} & L^p(\M) \ar[rr]^{(i^p_{(-\alpha)})^*}& &L^*_{(-\alpha)}\\
	&\M_* \ar[urr]_{(i^\infty_{(-\alpha)})^*}&}$$
which implies the following:
	\begin{align}\label{identify}
	& \text{\it For any $x\in L_{(\alpha)}$ we identify $i^p_{(\alpha)}x \in L^p(\M)$ for $1\le p \le \infty$}\\
	& \text{\it in the sense of interpolation.} \nonumber
	\end{align}
Now we focus a special case $\alpha = -\frac{1}{2}$.
From the definition \eqref{def-L-alpha} and the fact that $J\Delta^{-\frac{1}{2}} = \Delta^{\frac{1}{2}}J = S^*$ it is easy to check that 
for any $x \in \mathfrak{a}^2_0$ we have
	\begin{equation}\label{identify1}
	i^1_{(-\frac{1}{2})}(x)(y) = \varphi^{(-\frac{1}{2})}_x(y) = \varphi(yx)\;\, \text{for any}\;\, y\in \mathfrak{a}^2_0.
	\end{equation}
Moreover, we have an isometry (\cite[Lemma 5.4, Theorem 5.6]{Iz98})
	\begin{equation}\label{identify2}
	L^2(\M) \rightarrow \Hi_\varphi,\; i^2_{(-\frac{1}{2})}(x)\mapsto \Lambda x
	\end{equation}
for $x \in \mathfrak{a}^2_0$.
From now on, we will identify $L^2(\M)$ and $\Hi_\varphi$ using the above isometry. 

We close this subsection by considering an appropriate operator space structure on $L^p(VN(G))$, different from the usual convention.
Using the above interpolation result and Pisier's complex interpolation theory for operator spaces
we can endow an operator space structure on $L^p(\M)$ by
	\begin{equation}\label{OSS-Lp}
	{\mathcal O}L^p(\M) = [\M, \M^{op}_*]_{\frac{1}{p}}\;\, \text{(operator space sense)},
	\end{equation}
where $E^{op}$ implies the {\it opposite} of an operator space $E$ (\cite[section 2.10]{P03}).
Note that the operator space structure on $L^2(\M)$ given by the above is Pisier's operator Hilbert space (shortly $OH$) structure. 
However, the usual operator space structure on $A(G)$ is obtained by considering it as the predual of $VN(G)$.
Because of this disagreement we consider the following operator space structure on $L^p(\M)$.
	$$L^p(\M) = \left\{ \begin{array}{l} \mathcal{O}L^p(\M)^{op} \;\, \text{for}\;\, 1\le p \le 2 \\
	\mathcal{O}L^p(\M) \;\, \text{for}\;\, 2\le p \le \infty. \end{array} \right.$$
If we use reiteration theorem, then we get $\mathcal{O}L^p(\M) = [L^2(\M)_{oh}, \M^{op}_*]_{\frac{2}{p}-1}$ for $1\le p\le 2$,
where $H_{oh}$ is the operator Hilbert space defined on a Hilbert space $H$.
Since the opposite of $OH$ is still $OH$, we get
	\begin{equation}\label{OSS-Lp2}
	L^p(\M) = [L^2(\M)_{oh}, \M_*]_{\frac{2}{p}-1}\; \text{for}\; 1\le p\le 2.
	\end{equation}

\subsection{$A(G)$-module structures on $L^p(VN(G))$}

In order to consider $L^p(VN(G))$ we need a n.s.f. weight on $VN(G)$,
and we will use the Plancherel weight $\varphi$ (\cite[III.3.3]{Bla} or \cite[VII.3]{Ta03}) on $VN(G)$ given by
	$$\varphi(L_f) = f(e)\;\, \text{for any}\;\, f\in C_c(G).$$
We consider the full Tomita algebra $\mathfrak{a}_0$ associated with $\varphi$ as before. 
Recall the following two multiplication maps
	$$\pi_1 : A(G) \prt A(G) \rightarrow A(G),\; f\otimes g \mapsto f g$$
and
	$$\pi_\infty : A(G) \prt VN(G) \rightarrow VN(G),\; f\otimes L_g \mapsto L_{f g}.$$
Now we want to consider an appropriate module structure on $L_p(VN(G))$ using complex interpolation.
The natural candidate would be the interpolation of the above two module actions $\pi_1$ and $\pi_\infty$.
However, they are not compatible in the sense of interpolation.
Indeed, for any $g\in C_c(G)$, it is easy to check that $L_g \in VN(G)$ is identified with
$i^1_{(-\frac{1}{2})}(L_g) = \check{g} \in A(G)$ (recall \eqref{identify1}) in the sense of interpolation, where $\check{g}(x) = g(x^{-1})$.
But $\pi_1(f \otimes \check{g}) = f \check{g} \in A(G)$ and $\pi_\infty (f\otimes L_g) = L_{f g} \in VN(G)$.
Thus, we need to find some other module actions which are compatible in the sense of interpolation.
Actually, we can find appropriate module actions $\pi_2$ and $\pi'_2$ on $L^2(VN(G))$ based on \eqref{identify2},
which are compatible with $\pi_1$ and $\pi_\infty$, respectively, as follows.
	$$\pi_2 : A(G) \prt L^2(VN(G)) \rightarrow L^2(VN(G)),\; f \otimes \Lambda(L_{\check{g}}) \mapsto \Lambda(L_{\check{f} \check{g}})$$
and
	$$\pi'_2 : A(G) \prt L^2(VN(G)) \rightarrow L^2(VN(G)),\; f \otimes \Lambda(L_g) \mapsto \Lambda(L_{f g})$$
for any $f\in A(G)$ and $g\in C_c(G)$.
The difficulty lies in proving $\pi_2$ and $\pi'_2$ are actually complete contractions.

For the detailed discussion it would be better to use Kac algebraic (or locally compact quantum group) notations.
Let $\Gamma$ is the co-multiplication
	$$\Gamma : VN(G) \rightarrow VN(G) \bar{\otimes} VN(G),\; \lambda(s) \mapsto \lambda(s)\otimes \lambda(s).$$
The Plancherel weight $\varphi$ is known to be strongly left and right invariant (\cite{ES92}), which means that
	\begin{equation}\label{left-inv}
	I \otimes \varphi (\Gamma x) = \varphi \otimes I (\Gamma x) = \varphi(x)  1
	\end{equation}
for any $x\in n_\varphi$, where $1$ implies the identity in $VN(G)$.

Now we can read $\pi_2$ as follows.
	\begin{equation}\label{pi2}
	\pi_2 : A(G) \otimes L^2(VN(G)) \rightarrow L^2(VN(G)),\; \om \otimes \Lambda x \mapsto \pi_2(\om \otimes \Lambda x)
	\end{equation}
and for any $y\in n_\varphi$
	$$\left\langle \pi_2(\om \otimes \Lambda x), \overline{\Lambda y} \right\rangle = \om \otimes \varphi \, (\Gamma y^* (1 \otimes x)).$$
Here, we understand $A(G)$ as the collection of bounded normal functionals on $VN(G)$ instead of usual function description,
and $\overline{L^2(VN(G))}$ is the Banach space dual of $L^2(VN(G))$ using the duality
$\left\langle \Lambda x, \overline{\Lambda y} \right\rangle = \varphi(y^* x)$ for any $x,y\in n_\varphi$.
Indeed, for $x = L_{\check{g}}$ and $\om$, the functional defined by $\om(\cdot) = \left\langle f, \cdot\, \right\rangle$, we have
	\begin{align*}
	\om\otimes \varphi \,(\Gamma \lambda(s) (1\otimes x)) & = \om\otimes \varphi \,(\lambda(s) \otimes \lambda(s) L_{\check{g}} )\\
	& = f(s)g(s)\\
	& = \varphi(\lambda(s)L_{\check{f} \check{g}})
	\end{align*}
for any $s\in G$.

In order to show that $\pi_2$ is completely contractive we first need to consider other operator space structures on $L^2(VN(G))$,
namely the row and the column Hilbert space structure. For any Hilbert space $H$ we denote the row and the column Hilbert space on $H$ by
$H_r$ and $H_c$, respectively.

	\begin{prop}\label{prop-sigma2}
	For any locally compact group $G$, $\pi_2$ in \eqref{pi2} extends to the following complete contraction.
		$$\pi_{2,c} : A(G) \prt L^2(VN(G))_c \rightarrow L^2(VN(G))_c.$$
	\end{prop}
\begin{proof}
We focus on the adjoint of $\pi_{2,c}$, which is better for concrete calculations. We have
	$$\pi^*_{2,c} : \overline{L^2(VN(G))}_r \rightarrow CB(L^2(VN(G))_c, VN(G))$$
and for any $x, y\in n_\varphi$
	\begin{equation}\label{sigma2}
	\pi^*_{2,c}(\overline{\Lambda y}) (\Lambda x) = I\otimes \varphi \,(\Gamma y^* (1\otimes x)).
	\end{equation}
Consider any $n,m \in \n$ and any finite collections $(A_i) \subseteq M_n$ and $(B_k) \subseteq M_m$.
For any fixed $(x_k), (y_i) \subseteq n_\varphi$ we set $\Phi = I_n \otimes \pi^*_{2,c} (\sum_i A_i \otimes \overline{\Lambda y}_i)$.
Then by the Cauchy-Schwarz inequality for Hilbert $C^*$-modules we have
	\begin{align*}
	\lefteqn{\norm{I_m \otimes \Phi \Big(\sum_k B_k \otimes \Lambda x_k \Big)}_{M_m(M_n(VN(G)))}}\\
	& = \norm{I_m \otimes I_n \otimes I \otimes \varphi \Big( \sum_{i,k} (B_k \otimes A_i \otimes \Gamma y^*_i (1\otimes x_k)) \Big)}\\
	& = \norm{I_m \otimes I_n \otimes I \otimes \varphi
	\Big( \Big[\sum_i 1\otimes A_i \otimes \Gamma y^*_i \Big] \Big[\sum_k B_k \otimes 1 \otimes 1\otimes x_k \Big] \Big)}\\
	& \le \norm{I_m \otimes I_n \otimes I \otimes \varphi
	\Big( \Big[\sum_{i,j} 1\otimes A_i A^*_j \otimes \Gamma y^*_i y_j \Big] \Big)}^{\frac{1}{2}}\\
	& \;\;\;\; \times \norm{I_m \otimes I_n \otimes I \otimes \varphi
	\Big( \Big[\sum_{k,l} B^*_k B_l \otimes 1 \otimes 1\otimes x^*_k x_l \Big] \Big)}^{\frac{1}{2}}\\
	& = \norm{\sum_{i,j} \varphi(y^*_i y_j) A_i A^*_j}^{\frac{1}{2}}_{M_n}
	\times \;\;\norm{\sum_{k,l} \varphi(x^*_k x_l) B^*_k B_l}^{\frac{1}{2}}_{M_m}
	\end{align*}
The last equality comes from \eqref{left-inv}, the left invariance of $\varphi$.
If we choose $(x_k), (y_i)$ so that $(\Lambda x_k), (\Lambda y_i)$ are arbitrarily close to an orthonormal basis on $L^2(VN(G))$,
then we can conclude that $\pi^*_2$ is completely contractive from the above estimation.

\end{proof}

We need one more ingredient, namely the case of $L^2(VN(G))$ with the row Hilbert space structure.
Let $\widetilde{L^2}(VN(G))$ be the completion of $n^*_\varphi$ with the inner product
	$$\left\langle \widetilde{\Lambda} x, \widetilde{\Lambda} y \right\rangle_{\widetilde{}} := \varphi(xy^*),$$
where $\widetilde{\Lambda} : n^*_\varphi \hookrightarrow \widetilde{L^2}(VN(G))$ is the canonical injection.
Then the following map extends to an isometric isomorphism between $L^2(VN(G))$ and $\widetilde{L^2}(VN(G))$.
	$$j : L^2(VN(G)) \rightarrow \widetilde{L^2}(VN(G)), \; \Lambda x \mapsto \widetilde{\Lambda} \sigma^\varphi_{\frac{i}{2}}(x)$$
for any $x\in \mathfrak{a}_0$ (recall \eqref{analytic-elts}).
Indeed, we have for any $x, y\in \mathfrak{a}_0$
	$$\left\langle \widetilde{\Lambda} \sigma^\varphi_{\frac{i}{2}}(x), \widetilde{\Lambda} \sigma^\varphi_{\frac{i}{2}}(y) \right\rangle_{\widetilde{}}
	= \varphi(\sigma^\varphi_{\frac{i}{2}}(x) \sigma^\varphi_{\frac{i}{2}}(y)^*) = \varphi(y^* x) = \left\langle \Lambda x, \Lambda y \right\rangle$$
by \cite[Lem 3.18]{Ta03}.

Now we consider $\widetilde{\pi}_2$, the translation of $\pi_2$ to $\widetilde{L^2}(VN(G))$. Indeed,
	$$\widetilde{\pi}_2 : A(G) \otimes \widetilde{L^2}(VN(G)) \rightarrow \widetilde{L^2}(VN(G))$$
is given by
	\begin{align*}
	\left\langle \widetilde{\pi}_2 (\om \otimes \widetilde{\Lambda} x), \overline{\widetilde{\Lambda} y} \right\rangle
	& = \left\langle \pi_2 (\om \otimes \Lambda \sigma^\varphi_{-\frac{i}{2}}(x)), \overline{\Lambda \sigma^\varphi_{-\frac{i}{2}}(y)} \right\rangle\\
	& = \om \otimes \varphi\, (\Gamma \sigma^\varphi_{-\frac{i}{2}}(y)^*[1 \otimes \sigma^\varphi_{-\frac{i}{2}}(x)])\\
	& = \om \otimes \varphi\, ([(I\otimes \sigma^\varphi_{-\frac{i}{2}})\Gamma y]^*[1\otimes \sigma^\varphi_{-\frac{i}{2}}(x)])\\
	& = \om \otimes \varphi\, ((1\otimes x)\Gamma y^*)
	\end{align*}
for any $x, y\in \mathfrak{a}_0$. We have used a property of Kac algebras (\cite[Corollary 2.5.7]{ES92}) and \cite[Lem 3.18]{Ta03}
in the third and the last equalities, respectively.
Thus, we get
	$$\widetilde{\pi}^*_2(\overline{\widetilde{\Lambda} y})(\widetilde{\Lambda} x) = I\otimes \varphi\, ((1\otimes x)\Gamma y^*)$$
for any $x, y\in \mathfrak{a}_0$, in which the positions of $1\otimes x$ and $\Gamma y^*$ are reversed compared to \eqref{sigma2}.
By the same calculations as in Proposition \ref{prop-sigma2} we get that
	$$\widetilde{\pi}^*_2 : \overline{\widetilde{L^2}(VN(G))}_c \rightarrow CB(\widetilde{L^2}(VN(G))_r, VN(G))$$
is completely contractive.
If we remind that $L^2(VN(G))$ and $\widetilde{L^2}(VN(G))$ are isometric, then we get the following.
	\begin{prop}
	For any locally compact group $G$, $\pi_2$ in \eqref{pi2} extends to the following complete contraction.
		$$\pi_{2,r} : A(G) \prt L^2(VN(G))_r \rightarrow L^2(VN(G))_r.$$
	\end{prop}

If we consider the adjoint maps again, then we have two complete contractions
	$$\pi^*_{2,c} : \overline{L^2(VN(G))}_r \rightarrow CB(L^2(VN(G))_c, VN(G))$$
and
	$$\pi^*_{2,r} : \overline{L^2(VN(G))}_c \rightarrow CB(L^2(VN(G))_r, VN(G)).$$
Since, for any Hilbert space $H$, $[H_r, H_c]_\frac{1}{2} = H_{oh}$,
where $[\,\cdot\, , \, \cdot\,]$ is the complex interpolation in the category of operator spaces and 
	$$[CB(E_0, F), CB(E_1, F)]_\frac{1}{2} \hookrightarrow CB([E_0, E_1]_\frac{1}{2}, F)$$
completely contractively for any operator spaces $E_0$, $E_1$, and $F$, we get a complete contraction
	$$\pi^*_2 : \overline{L^2(VN(G))}_{oh} \rightarrow CB(L^2(VN(G))_{oh}, VN(G))$$
by complex interpolation.
Furthermore, if we interpolate $\pi^*_1$ and $\pi^*_2$ we get the natural $A(G)$-module actions on $L^p(VN(G))$ ($1\le p\le 2$).

	\begin{thm}\label{thm-Lp-mod1}
	For any locally compact group $G$ and $1\le p\le 2$
 the following map is a complete contraction, which gives a left operator $A(G)$-module structure on $L^p(VN(G))$.
		$$\pi_p : A(G) \prt L^p(VN(G)) \rightarrow L^p(VN(G))$$
	and $\pi_p(f \otimes i^p(L_{\check{g}})) = i^p(L_{\check{f}\check{g}})$ for any $f\in A(G)$ and $g\in C_c(G)$,
	where $i^p = i^p_{(-\frac{1}{2})}$ in \eqref{identify}.	
	\end{thm}

Now we move to the case $p \ge 2$ and focus on $\pi'_2$, which is compatible with $\pi_\infty$.
Using the above Kac algebraic notation we can read $\pi'_2$ as follows.
	\begin{equation}\label{pi2'}
	\pi'_2 : A(G) \otimes L^2(VN(G)) \rightarrow L^2(VN(G)),\; \om \otimes \Lambda x \mapsto \Lambda \otimes \om\, (\Gamma x)
	\end{equation}
for any $\om \in A(G)$ and $x \in n_\varphi$.
If we consider the adjoint map again, then we have
	$$\pi'^*_2 : \overline{L^2(VN(G))}_r \rightarrow CB(L^2(VN(G))_c, VN(G)).$$
Thus, for any $x,y \in n_\varphi$ we have
	$$\pi'^*_2(\overline{\Lambda y})(\Lambda x) = \varphi \otimes I ((y^*\otimes 1)\Gamma x).$$
If we repeat all the above calculations using \eqref{left-inv}, the right invariance of $\varphi$, then we get the following similar results.
	\begin{thm}\label{thm-Lp-mod2}	
	For any locally compact group $G$, $\pi'_2$ in \eqref{pi2'} extend to the following complete contractions.
		$$\pi'_{2,c} : A(G) \prt L^2(VN(G))_c \rightarrow L^2(VN(G))_c$$
	and
		$$\pi'_{2,r} : A(G) \prt L^2(VN(G))_r \rightarrow L^2(VN(G))_r.$$
	Moreover, for $2\le p < \infty$, we get the following complete contraction, which gives a left completely bounded $A(G)$-module structure on $L^p(VN(G))$.
		$$\pi'_p : A(G) \prt L^p(VN(G)) \rightarrow L^p(VN(G)),$$
	where $\pi'_p(f \otimes i^p(L_g)) = i^p(L_{f g})$ for any $f\in A(G)$ and $g\in C_c(G)$.
	\end{thm}

\subsection{Operator Projectivity of $L^p(VN(G))$ for $1<p<\infty$}\label{noncomLp}

	\begin{lem}\label{lem-essential}
	Let $G$ be a discrete group. Then $L^p(VN(G))$ ($1 < p <\infty$) is an essential $A(G)$-module.
	\end{lem}
\begin{proof}
Since we have $[A(G), C^*_r(G)]_{\frac{1}{p}} = L^p(VN(G))$ (\cite[Proposition 3.1]{JR})
we know $C^*_r(G) = A(G) \cap C^*_r(G)$ is dense in $L^p(VN(G))$ (\cite[Theorem 4.2.2]{BL76}).
Thus, $L^p(VN(G))$ is essential for $1 < p < \infty$ since $\pi_p(f \otimes i^p(L_{\check{g}})) = i^p(L_{\check{f} \check{g}})$ ($1< p \le 2$)
and $\pi'_p(f \otimes i^p(L_g)) = i^p(L_{f g})$ ($2\le p <\infty$) for any $f\in A(G)$ and $g\in C_c(G)$.
\end{proof}

Since it is well known that $A(G)$ is operator biprojective when $G$ is discrete (\cite{W02}) we get the following by Theorem \ref{thm-essential}.

	\begin{thm}\label{thm-positive-Lp}
	If $G$ is a discrete and amenable group, then $L^p(VN(G))$ ($1<p<\infty$) is operator projective in $A(G)$-{\bf mod}.
	\end{thm}

For $2\le p\le \infty$ we have a better description of $\pi'_p$. Let
	$$\lambda : L^1(G) \rightarrow VN(G) \subseteq B(L^2(G)),\, f\mapsto L_f$$
be the left regular representation of $L^1(G)$.
Then we can consider the unitary map $\lambda_2 : L^2(G) \rightarrow L^2(VN(G))$, which is defined by
	$$\lambda_2(f) = \Lambda (L_f)$$
for any $f\in L^1(G) \cap L^2(G)$.
Since $\lambda$ and $\lambda_2$ are compatible in the sense of interpolation we get complete contractions
	$$\lambda_p : L^p(G) \rightarrow L^{p'}(VN(G))$$
for $1\le p\le 2$ by complex interpolation,
where $\frac{1}{p} + \frac{1}{p'} = 1$ and $\lambda_p(f) = i^{p'}(L_f)$ for any $f \in L^1(G) \cap L^2(G)$.
Now we can describe $\pi'_p$ for $2\le p \le \infty$ as follows.
	$$\pi'_p(f \otimes \lambda_{p'}(g)) = \lambda_{p'}(f g)$$ for any $f\in A(G)$ and $g\in L^{p'}(G)$, where $\frac{1}{p} + \frac{1}{p'} = 1$.

We also need the following transference result.
	\begin{lem}\label{lem-opensubgp-Lp}
	Let $H$ be an open subgroup of a locally compact group $G$ and $1<p<\infty$.
	If $L^p(VN(G))$ is operator projective in $A(G)$-{\bf mod}, then $L^p(VN(H))$ is operator projective in $A(H)$-{\bf mod}.
	\end{lem}
\begin{proof}
We interpolate the following two completely contractive projections from Lemma \ref{lem-opensubgp}
	$$R_1 : A(G) \rightarrow A(H),\, g \mapsto g|_H,$$
and
	$$R_\infty : VN(G) \rightarrow VN(H),\, L_f \mapsto L_{f|_H}.$$
Note that $R_1(\check{g}) = \check{g}|_H$, so that $R_1$ and $R_\infty$ are compatible in the sense of interpolation.
Then by complex interpolation we get a completely contractive projection
	$$R_p : L^p(VN(G)) \rightarrow L^p(VN(H)),\, i^p(L_f) \mapsto i^p(L_{f|_H})$$
for $1<p<\infty$.
Similarly we get a completely contractive inclusion $$j_p : L^p(VN(H)) \rightarrow L^p(VN(G)), i^p(L_f) \mapsto i^p(L_{\widetilde{f}}),$$
where $\widetilde{f}$ is the extension of $f$ to $G$ by assigning 0 outside of $H$. Indeed, $j_p$ is a complete isometry since $R_p\circ j_p = I$.

Thus, by following a diagram similar to \eqref{diagram1} we get the conclusion.

\end{proof}	

	\begin{prop}\label{prop-trivial-mapping-space-Lp}
	Let $G$ be a non-discrete group. Then
		$${}_{A(G)}B(L^p(VN(G)), A(G)) = 0$$
	for $2\le p< \infty$.
	\end{prop}
	\begin{proof}
The proof is basically the same as in Proposition \ref{prop-trivial-mapping-space}.
We pick any $T \in {}_{A(G)}B(L^p(VN(G)), A(G))$ and any non-zero $h \in C_c(G)$. 
Then since $G$ is not discrete, $K = \{h\neq 0\}$ is an open set with compact closure and positive measure.
Then, with $g = T(\lambda_{p'}(1_K))$ we can similarly show that
	$$T(\lambda_{p'}(f 1_K)) = f T(\lambda_{p'}(1_K)) = f g\;\text{a.e.}$$
for any $f\in L^{p'}(G)$.

Now we consider the following composition of bounded maps.
	$$\Phi : L^{p'}(G) \stackrel{\lambda_p}{\longrightarrow} L_p(VN(G)) \stackrel{T}{\longrightarrow} A(G) \stackrel{j}{\hookrightarrow} L^{\infty}(G),$$
where $j$ is canonical embedding. Then, we have
	$$\Phi(f 1_K) = f g$$
for any $f\in C_c(G)$ with $\text{supp}(f) \subseteq K$. As in the proof of \ref{prop-trivial-mapping-space} we can show that $g=0$.
Thus, $T(\lambda_{p'}(h)) = T(\lambda_{p'}(h 1_K)) = hg = 0$, and by a standard density argument we have $T = 0$, so that
	$${}_{A(G)}B(L^p(VN(G)), A(G)) = 0.$$
\end{proof}

By a similar argument as in Theorem \ref{thm-notproj-Cr(G)} we get the following.
	\begin{thm}
	Let $G$ be a non-discrete locally compact group.
	Then $L^p(VN(G))$ $(2\le p <\infty)$ is not operator projective in $A(G)$-{\bf mod}.
	\end{thm}

Again, for the case of discrete groups containing $\mathbb{F}_2$, we have negative results.
Recall that $RC^p_n$ is the $n$-dimensional Hilbertian operator space (\cite{P03}) with standard basis $(e_i)^n_{i=1}$ defined by
	$$RC^p_n = [R_n\cap C_n, R_n+C_n]_{\frac{1}{p}}\; \text{for}\; 1\le p \le \infty.$$
Note that $(RC^p_n)^* = RC^{p'}_n$ for $\frac{1}{p} + \frac{1}{p'} = 1$.

	\begin{thm}\label{thm-free-Lp}
	Let $G$ be a discrete group containing $\mathbb{F}_2$.
	Then, $L^p(VN(G))$ $(1<p<\infty)$ is not operator projective in $A(G)$-{\bf mod}.
	\end{thm}
\begin{proof}
The proof is essentially the same as Theorem \ref{thm-free-Cr(G)}.
Again, it is enough to show the theorem for the case $G = \mathbb{F}_\infty$
because of Lemma \ref{lem-opensubgp-Lp} and the fact that $\mathbb{F}_2$ contains $\mathbb{F}_\infty$.
Note that since $VN(G)$ is a finite von Neumann algebra we have $VN(G) \hookrightarrow L^p(VN(G))$ for any $1\le p<\infty$,
and since $C^*_r(G)$ is dense in $L^p(VN(G))$ (see the proof of Lemma \ref{lem-essential}),
we can say that $\lambda(s)$, $s\in G$ is a typical element in $L^p(VN(G))$.

Suppose that $L^p(VN(G))$ is operator projective in $A(G)$-{\bf mod}. Then we have a completely bounded left $A(G)$-module map
	$\rho : L^p(VN(G)) \rightarrow A(G) \prt L^p(VN(G)),$
which is a right inverse of the multiplication map $\pi$. We can similarly show that $\rho$ is of form
$\rho(\lambda(s)) = \delta_s \otimes (\lambda(s) + x_s)$
for any $s\in G$ and for some $x_s \in X$ with $L^p(VN(G)) = X \oplus \mathbb{C}\lambda(s)$.

Now we set
	$$E^p_n = \text{span}\{\lambda(g_i)\}^n_{i=1} \subseteq L^p(VN(G))$$
and
	$$P^p_n : L^p(VN(G)) \rightarrow E^p_n,\;\lambda(g_i) \mapsto
	\left\{ \begin{array}{ll} \lambda(g_i) & 1\le i\le n\\ 0 & \text{otherwise} \end{array}\right.,$$
Note that we have complete 2-isomorphisms
	$$\phi_p : E^p_n \rightarrow RC^p_n,\; \lambda(g_i) \mapsto e_i,$$
and $P^p_n$ is completely bounded with cb-norm $\le 2$ (\cite[section 9.7]{P03}).
Recall $\psi$ and $P^1_n$ as in \eqref{E1} and \eqref{P1}, respectively.
By composing the above maps with $\rho$ we get the following completely bounded map with cb-norm independent of $n$: 
	$$\Phi : RC^p_n \stackrel{\phi^{-1}_p}{\longrightarrow} L^p(VN(G)) \stackrel{\rho}{\longrightarrow} A(G) \prt L^p(VN(G))
	\stackrel{(\psi \otimes \phi)\circ (P^1_n\otimes P^p_n)}{\longrightarrow} (RC^1_n) \prt (RC^p_n),$$
where $\Phi(e_i) = e_i \otimes (e_i + \sum_{1\le j\neq i\le n}\alpha_{i,j}e_j)$ for some constants $(\alpha_{i,j})_{1\le j\neq i\le n}$.
Note that
	$$RC^1_n \prt RC^p_n \cong CB(RC^p_n, RC^\infty_n)^*$$
completely isometrically, so that we have
	$$\norm{\Phi \Big(\sum^n_{i=1}e_i\Big)}_{CB(RC^p_n, RC^\infty_n)^*}
	\ge \frac{\abs{\left\langle \Phi\Big(\sum^n_{i=1}e_i\Big), I_n \right\rangle}}{\norm{I_n : RC^p_n \rightarrow RC^\infty_n}_{cb}}
	\le n^{1-\frac{1}{2p}},$$
where $I_n : RC^p_n \rightarrow RC^\infty_n$ is the formal identity.
Note that
	$$\norm{I_n : RC^p_n \rightarrow RC^\infty_n}_{cb} \le n^{\frac{1}{2p}}$$
can be obtained by complex interpolation and the facts that
	$$RC^p_n = [RC^\infty_n, RC^2_n]_{\frac{2}{p}}$$
and (\cite[section 10]{P03})
	$$\norm{I : RC^2_n \rightarrow RC^\infty_n}_{cb} \le n^{\frac{1}{4}}.$$
On the other hand we have
	$$\norm{\Phi \Big(\sum^n_{i=1}e_i\Big)}_{CB(RC^p_n, RC^\infty_n)^*} \le \norm{\Phi} \norm{\sum^n_{i=1}e_i}_{RC^p_n} = \norm{\Phi}\sqrt{n},$$
which is a contradiction when $n$ is large enough.

\end{proof}

We close this section with the case $L^2(VN(G))_c$.
Note that $\pi_{2,c}$ in Proposition \ref{prop-sigma2} gives us a completely contractive left $A(G)$-module structure on $L^2(VN(G))_c$.
However, operator projectivity of $L^2(VN(G))_c$ has nothing to do with amenability
in contrast to the case of $L^2(VN(G))$ (with operator Hilbert space structure).

	\begin{thm}\label{thm-L2column}
	Let $G$ be a locally compact group. Then $L^2(VN(G))_c$ is operator projective in $A(G)$-{\bf mod}
	if and only if $G$ is discrete.
	\end{thm}
\begin{proof}
When $G$ is non-discrete we can use the same argument as in Theorem \ref{thm-notproj-Cr(G)} to show that $L^2(VN(G))_c$ is not operator projective.
Conversely, suppose that $G$ is discrete. Let
	$$\Gamma_2 : L^2(VN(G)) \rightarrow L^2(VN(G)) \otimes_2 L^2(VN(G)),\; \Lambda x \mapsto \Lambda \otimes \Lambda (\Gamma x)$$
for any $x \in n_\varphi.$ Then, from the left invariance of $\varphi$ and the fact that $G$ is discrete,
we can conclude that $\Gamma_2$ is an isometry, so that
	$$\Gamma_2 : L^2(VN(G))_c \rightarrow L^2(VN(G))_c \prt L^2(VN(G))_c$$ is a complete isometry.
Note that the multiplication map
	$$\Phi : L^2(VN(G))_c \prt L^2(VN(G))_r \rightarrow L^1(VN(G)) = L^\infty(\Gb)_*, \; a\otimes b \mapsto ab$$
is a complete contraction. Then, we get a complete contraction
	$$\phi = \Phi(\,\cdot\, , \Lambda 1_{VN(G)}) : L^2(VN(G))_c \rightarrow L^1(VN(G)), \; \Lambda x \mapsto \varphi^{(-\frac{1}{2})}_x.$$
Then, the composition $\rho = (\phi \otimes I)\circ \Gamma_2 : L^2(VN(G))_c \rightarrow L^1(VN(G)) \prt L^2(VN(G))_c$
is completely contractive and it is a $A(G)$-module map. Moreover it is straightforward to check that $\rho$ is a right inverse of $\pi_{2,c}$.
\end{proof}

\section{The modules $L^p(G)$ $(1\le p \le \infty)$, $C^*(G)$, and $C^*_{\delta}(G)$}

\subsection{Operator Projectivity of $L^p(G)$ $(1\le p \le \infty)$}

In this subsection we will again assume that the reader is familiar with standard materials about
complex interpolation of Banach spaces (\cite{BL76}) and operator spaces (\cite{P96}).
For $L^p(G)$ we consider the natural operator space structure given by complex interpolation
and the fact $L^p(G) = [L^\infty(G), L^1(G)]_{\frac{1}{p}}$.
The $A(G)$-module structure on $L^p(G)$ $(1\le p \le \infty)$ can be obtained by the following complete contraction.
	\begin{equation}\label{module-Lp}
	\pi : A(G) \prt L^p(G) \rightarrow L^p(G),\; f\otimes g \mapsto f g.
	\end{equation}
Indeed, since $L^\infty(G)$ is a completely contractive Banach algebra under the pointwise multiplication, we have a complete contraction
	$$\pi_\infty : L^\infty(G) \prt L^\infty(G) \rightarrow L^\infty(G),\; f\otimes g \mapsto f g.$$
On the other hand $L^1(G)$ is the predual of $L^\infty(G)$,
so that it is a $L^\infty(G)$-module with the following completely contractive multiplication map
	$$\pi_1 : L^\infty(G) \prt L^1(G) \rightarrow L^1(G),\; f\otimes g \mapsto f g.$$
Then, we get a complete contraction
	$$\pi_p : L^\infty(G) \prt L^p(G) \rightarrow L^p(G),\; f\otimes g \mapsto f g$$
by bilinear complex interpolation (\cite[section 2.7]{P03}).
Since the formal identity $i : A(G) \rightarrow L^\infty(G)$ is a complete contraction,
we get \eqref{module-Lp} by the composition $\pi_p \circ (i \otimes id_{L^p(G)})$.
Clearly, $L^p(G)$ is essential for $1\le p<\infty$.
Then, by similar arguments as in the proof of Theorem \ref{thm-notproj-Cr(G)}, we get the following result.
Note that the transference argument can be obtained by a similar diagram to \eqref{diagram1}.

	\begin{thm}\label{thm-comLp}
	Let $G$ be a locally compact group. If $G$ is discrete and amenable,
	then $L^p(G)$ $(1\le p<\infty)$ is operator projective in $A(G)$-{\bf mod}.
	If $G$ is non-discrete, then $L^p(G)$ $(1\le p < \infty)$ is not operator projective in $A(G)$-{\bf mod}.
	\end{thm}

When $p=1$ we have more positive results.
	\begin{thm}\label{thm-comL1}
	Let $G$ be a locally compact group. If $G$ is discrete,
	then $\ell^1(G)$ is operator projective in $A(G)$-{\bf mod}.
	\end{thm}
\begin{proof}
Let
	$$\rho : \ell^1(G) \rightarrow A(G) \prt \ell^1(G),\; \delta_s \mapsto \delta_s \otimes \delta_s.$$
It is straightforward to check that $\rho$ is a left $A(G)$-module map, which is a right inverse of the multiplication map, provided that it is completely bounded.
Indeed, for any finite subset $\{g_1, \cdots, g_n\} \subseteq G$ we have
	\begin{align*}
	\norm{\rho \Big(\sum_i \alpha_i \delta_{g_i}\Big)}_{A(G) \prt \,\ell^1(G)}
	& = \norm{\sum_i \alpha_i \delta_{g_i}\otimes \delta_{g_i}}_{A(G) \prt \, \ell^1(G)}\\
	& = \sum_i \abs{\alpha_i}\cdot \norm{\delta_{g_i}}_{A(G)}\\
	& \le \sum_i \abs{\alpha_i},
	\end{align*}
which implies that $\rho$ is contractive.
Since $\ell^1(G)$ is equipped with the maximal operator space structure, $\rho$ is actually a complete contraction. 
Consequently, $\ell^1(G)$ is operator projective by Proposition \ref{prop-woods}.

\end{proof}

When $p=\infty$ we have negative results for discrete groups.
	\begin{lem}\label{lem-proj-com}
	Let $G$ be a discrete group. Suppose that $\ell^\infty(G)$ is operator projective in $A(G)$-{\bf mod}.
	Then there is a bounded projection from $\ell^\infty(G)$ onto $c_0(G)$.
	\end{lem}  
\begin{proof}
Note that $A(G) \cdot \ell^\infty(G) \subseteq c_0(G)$ when $G$ is discrete.
Moreover, we have a b.a.i. $(e_i)_i$ in $(c_0(G), \norm{\cdot}_\infty)$ satisfying $(e_i)_i \subseteq A(G)$.
Then, the rest of the proof is the same as that of \cite[Lemma 3.2]{DP04}.
\end{proof}

By combining the above lemma and \cite[Corollary 3]{LL90} we have the following.
	\begin{thm}\label{thm-ell-infty}
	Let $G$ be a discrete group. Then $\ell^\infty(G)$ is operator projective in $A(G)$-{\bf mod} if and only if $G$ is finite.
	\end{thm}

Again, for the case of discrete groups containing $\mathbb{F}_2$, we have negative results.
	\begin{thm}\label{thm-free-comLp}
	Let $G$ be a discrete group containing $\mathbb{F}_2$.
	Then $\ell^p(G)$ $(1<p < \infty)$ is not operator projective in $A(G)$-{\bf mod}.
	\end{thm}
\begin{proof}
The proof is essentially the same as Theorem \ref{thm-free-Cr(G)}.
Again, it is enough to show the theorem for the case $G = \mathbb{F}_\infty$
because of Lemma \ref{lem-opensubgp-Lp} and the fact that $\mathbb{F}_2$ contains $\mathbb{F}_\infty$.

Suppose that $\ell^p(G)$ is operator projective. Then we have a completely bounded left $A(G)$-module map
	$\rho : \ell^p(G) \rightarrow A(G) \prt \ell^p(G),$
which is a right inverse of the multiplication map $\pi$. We can similarly show that $\rho$ is of form
$\rho(\delta_s) = \delta_s \otimes (\delta_s + x_s)$
for any $s\in G$ and for some $x_s \in X$ with $\ell^p(G)= X \oplus \mathbb{C}\delta_s$.

Now we set
	$$F^p_n = \text{span}\{\delta_{g_i}\}^n_{i=1} \subseteq \ell^p(G)$$
and
	$$Q^p_n : \ell^p(G) \rightarrow F^p_n,\; \delta_{g_i}\mapsto
	\left\{ \begin{array}{ll} \delta_{g_i} & 1\le i\le n\\ 0 & \text{otherwise} \end{array}\right..$$
Clearly $\phi_p : F^p_n \rightarrow \ell^p_n,\; \delta_{g_i} \mapsto e_i$ is a complete contraction,
and $Q^p_n$ is also a complete contraction. Recall the maps $\psi$ and $P^1_n$ in \eqref{E1} and \eqref{P1}, respectively.
By composing the above maps with $\rho$ we get the following completely bounded map with cb-norm independent of $n$: 
	$$\Phi : \ell^p_n \stackrel{\phi^{-1}}{\longrightarrow} \ell^p(G) \stackrel{\rho}{\longrightarrow} A(G) \prt \ell^p(G)
	\stackrel{(\psi \otimes \phi)\circ (P^1_n\otimes Q^p_n)}{\longrightarrow} (RC^1_n) \prt (\ell^p_n),$$
where $\Phi(e_i) = e_i \otimes (e_i + \sum_{1\le j\neq i\le n}\alpha_{i,j}e_j)$ for some constants $(\alpha_{i,j})_{1\le j\neq i\le n}$.
Note that
	$$RC^1_n \prt \ell^p_n \cong CB(\ell^p_n, RC^\infty_n)^*$$
completely isometrically, so that we have
	$$\norm{\Phi \Big(\sum^n_{i=1}e_i\Big)}_{CB(\ell^p_n, RC^\infty_n)^*}
	\ge \frac{\abs{\left\langle \Phi\Big(\sum^n_{i=1}e_i\Big), I_n \right\rangle}}{\norm{I_n : \ell^p_n \rightarrow RC^\infty_n}_{cb}}
	= n^{\frac{1}{2} + \frac{1}{2p}}$$
by Lemma \ref{lem-cb-norm-estimate} below.

On the other hand we have
	$$\norm{\Phi \Big(\sum^n_{i=1}e_i\Big)}_{CB(\ell^p_n, RC^\infty_n)^*}
	\le \norm{\Phi} \norm{\sum^n_{i=1}e_i}_{\ell^p_n} = \norm{\Phi}n^{\frac{1}{p}},$$
which is a contradiction when $n$ is large enough.

\end{proof}

	\begin{lem}\label{lem-cb-norm-estimate}
	Let $I_n :\ell^p_n \rightarrow RC^\infty_n$ be the formal identity for $1\le p \le \infty$. Then
		$$\norm{I_n}_{cb} \le n^{\frac{1}{2}-\frac{1}{2p}}.$$
	\end{lem}
\begin{proof}
We will first check the extremal cases $p=1$ and $p=\infty$.
When $p=1$ we have
	$$\norm{I_n :\ell^1_n \rightarrow RC^\infty_n}_{cb} = \norm{I_n :\ell^1_n \rightarrow \ell^2_n} \le 1.$$
When $p=\infty$ we have for any $(A_i)^n_{i=1} \subseteq M_m$, $m\in \n$ that
	\begin{align*}
	\norm{\sum^n_{i=1} A_i \otimes e_i}_{M_m(RC^\infty_n)}
	& = \max \Bigg\{ \norm{\sum^n_{i=1} A^*_i A_i}^{\frac{1}{2}}_{M_m}, \norm{\sum^n_{i=1} A_i A^*_i}^{\frac{1}{2}}_{M_m}\Bigg\}\\
	& \le n^{\frac{1}{2}}\max_{1\le i\le n} \norm{A_i}_{M_m} = n^{\frac{1}{2}}\norm{\sum^n_{i=1} A_i \otimes e_i}_{M_m(\ell^\infty_n)},
	\end{align*}
which implies
	$$\norm{I_n :\ell^\infty_n \rightarrow RC^\infty_n}_{cb} \le n^{\frac{1}{2}}.$$
The final result follows by applying complex interpolation.
\end{proof}

\subsection{Operator Projectivity of $C^*(G)$}

Let $i : L^1(G) \hookrightarrow C^*(G)$ is the canonical embedding. Then $\{i(f) : f\in L^1(G)\}$ is dense in $C^*(G)$.
Now we can consider a natural $A(G)$-module structure on $C^*(G)$ obtained by the following complete contraction.
	$$\pi : A(G) \prt C^*(G) \rightarrow C^*(G),\; f \otimes i(g) \mapsto i(f g).$$
Indeed, $(C^*(G))^* = B(G)$ and $B(G)$ is a completely contractive $A(G)$-module under the pointwise multiplication,
then we can easily check that $\pi$ is the dual module structure.
It is clear that $C^*(G)$ is essential.
Then, by Theorem \ref{thm-essential} and a similar argument as in the proof of Theorem \ref{thm-notproj-Cr(G)}, we get the following result.

	\begin{thm}\label{thm-fullC*}
	Let $G$ be a locally compact group. If $G$ is discrete and amenable,
	then $C^*(G)$ is operator projective in $A(G)$-{\bf mod}.
	If $G$ is non-discrete, then $C^*(G)$ is not operator projective in $A(G)$-{\bf mod}.
	\end{thm}

	\begin{rem}{\rm 
	Unfortunately, the argument for $\mathbb{F}_\infty$ break down in this case since the span of free generators is completely isomorphic to $\ell^1_n$
	and has complementation constant order of $\sqrt{n}$.
	}
	\end{rem}

\subsection{Operator Projectivity of $C^*_\delta(G)$}
We close this section with the case of $C^*_\delta(G)$, the $C^*$-algebra generated by $\{\lambda(s) : s\in G\} \subseteq B(L^2(G))$.
If $G$ is discrete, then $C^*_\delta(G) = C^*_r(G)$, so that the only relevant case is when $G$ is non-discrete. 

	\begin{prop}\label{prop-trivial-mapping-space-C*delta}
	Let $G$ be a non-discrete group. Then
		$${}_{A(G)}B(C^*_\delta(G), A(G)) = 0.$$
	\end{prop}
	\begin{proof}
We pick any $T \in {}_{A(G)}B(C^*_\delta(G), A(G))$ and $s\in G$. Then
	$$T(\lambda(s)) = T(f\cdot \lambda(s)) = f T(\lambda(s))$$
for any $f\in A(G)$ with $f(s) = 1$.
If we pick $f \in A(G)$ such that $f(s) = 1$ and $f$ is supported in an arbitrarily small neighborhood of $s$,
then we can observe that $T(\lambda(s))$ is zero in $G \backslash \{s\}$, and consequently $T(\lambda(s)) = 0$.
Then by density we get $T = 0$. 
\end{proof}

Then, by similar arguments as in the proof of Theorem \ref{thm-notproj-Cr(G)}, we get the following result.
Note that the transference argument can be obtained by a similar diagram to \eqref{diagram1}.

	\begin{thm}\label{thm-C*delta}
	Let $G$ be a locally compact group. If $G$ is non-discrete, then $C^*_\delta(G)$ is not operator projective in $A(G)$-{\bf mod}.
	\end{thm}

\section{Open problems and Summary of the results}
We collect some open problems which we could not address:
\vspace{0.5cm}

Q1. Is there a locally compact group $G$ such that $\dim(B(G)/A(G))=1?$
\vspace{0.3cm}

Note that if the answer is yes, then it follows from Lemma \ref{L:co-dim-sigma-compact} that 
$G_e$ is compact. Thus $H=G / G_e$ have the same property. That is, by Theorem \ref{T:finite-co-dim},
there exists a non-compact, non-discrete, totally disconnected group $H$ such that 
$\dim(B(H)/A(H))=1$. Moreover, $H$ is minimally almost periodic.
\vspace{0.5cm}

Q2. Is $C^*(\mathbb{F}_2)$ operator projective in $A(\mathbb{F}_2)$-\m?
\vspace{0.5cm}

Q3. Is $L^\infty(G)$ operator projective in $A(G)$-{\bf mod} for an infinite non-discrete group $G$?
\vspace{0.5cm}

Q4. Is $L^p(VN(G))$ operator projective in $A(G)$-{\bf mod} for a non-discrete group $G$ and $1<p<2$?
\vspace{0.5cm}

We will end the paper with table that contains a summary of our results. The first column in the table denotes the module 
under consideration. The second column identifies those classes of groups for which we know 
definitively that the module under consideration is operator projective as a left $A(G)$-module.
The third column identifies those classes of groups for which we know 
definitively that the module under consideration is not operator projective as a left $A(G)$-module.
The fourth column summaries what we know in the specific case 
that $G$ is a discrete group containing $\mathbb{F}_2$.
\vspace{0.5cm}
  
\begin{table}
\caption{Summary of the results}
\begin{tabular}{|p{2cm}|p{2cm}|p{3cm}||p{3cm}|}

\hline $A(G)$-{\bf mod} & op. proj. & not op. proj. & $G$ is a discrete group containing $\mathbb{F}_2$ \\

\hline $A(G)$ & [IN] & $G = SL(3,\mathbb{R})$ & [DS] $\subseteq$ [IN]\\

\hline $C^*_r(G)$, $UCB(\widehat{G})$ & (1) & (2) & not op. proj. \\

\hline $L^p(VN(G))$ $(1<p<\infty)$ & (1) & (2) \;\;\; when \;\;\;\;$2\le p <\infty$ & not op. proj. \\

\hline $L^2(VN(G))_c$ & $G$ is discrete & (2) & op. proj. \\

\hline $VN(G)$ & $G$ is finite & (2) or(4) & not op. proj. \\

\hline $C^*(G)$ & (1) & (2) & ? \\

\hline $L^1(G)$ & $G$ is discrete & (2) & op. proj. \\

\hline $L^p(G)$ $\;\;\;(1<p<\infty)$ & (1) & (2) & not op. proj. \\

\hline $L^\infty(G)$ & $G$ is finite & $G$ is infinite and discrete & not op. proj. \\

\hline $C^*_\delta (G)$ & (1) & (2) & $C^*_\delta(G) = C^*_r(G)$ \\

\hline $B(G)$ & (3) & not (3) & not op. proj.\\

\hline $A(G)^{**}$, $UCB(\widehat{G})^*$ & $G$ is finite & (2) or (4) & not op. proj.\\ 

\hline \multicolumn{4}{l}{(1) : $G$ is discrete and amenable.}\\

\hline \multicolumn{4}{l}{(2) : $G$ is non-discrete.}\\

\hline \multicolumn{4}{l}{(3) : $G$ is compact or $\text{dim}(B(G)/A(G)) = 1$.}\\

\hline \multicolumn{4}{l}{(4) : $G$ contains an infinite amenable subgroup.}\\

\hline
\end{tabular}

\end{table}

\bibliographystyle{amsplain}
%\bibliography{hidden}

\begin{thebibliography}{1}

\bibitem{ABU} C. A. Akemann, B. Tanbay, and A. $\ddot {U}$lger, A note on the Kadison-Singer Problem, J. Operator Theory, (to appear).  

\bibitem{Arsac}
G. Arsac, Sur l'espace de Banach engendr$\acute{\text{e}}$ par les coefficients d'une repr$\acute{\text{e}}$sentation unitaire,
Publ. D$\acute{\text{e}}$p. Math (Lyon) \textbf{13} (1976), 1-101.

\bibitem{Ar02}
O. Y. Aristov, Biprojective algebras and operator spaces. Functional analysis, 8.  J. Math. Sci. (New York) \textbf{111} (2002),  no. 2, 3339-3386.

\bibitem{Ar05}
O. Y. Aristov, Fourier algebras of some connected groups are not projective. (Russian)  Uspekhi Mat. Nauk  \textbf{60} (2005),  no. 1(361), 159-160;  translation in Russian Math. Surveys \textbf{60} (2005), no. 1, 154-156.

\bibitem{BKLS} M. E. Bekka, E. Kaniuth, A. T. Lau, and G. Schlichting, Weak$^*$-closedness of subspaces
of Fourier-Stieltjes algebras and weak$^*$-continuity of the restriction map. Trans. Amer. Math. Soc.,
\textbf{350} (1998), no. 6, 2277-2296.

\bibitem{BelFor}
A. B$\acute{\text{e}}$langer and B. Forrest,
Goemetric properties of coefficient function spaces determined by unitary representations of a locally compact group.
J. Math. Anal. Appl., \textbf{193} (1995), 390-405.

\bibitem{Bla}
B. Blackadar, Operator algebras. Theory of $C\sp *$-algebras and von Neumann algebras. Encyclopaedia of Mathematical Sciences, 122.
Operator Algebras and Non-commutative Geometry, III. Springer-Verlag, Berlin, 2006. xx+517 pp.

\bibitem{BL76}
J. Bergh and J. L\"{o}fstr\"{o}m, Interpolation spaces. Springer-Verlag, Berlin, 1976.

\bibitem{D}
H. G. Dales, Banach algebras and automatic continuity, New York, Oxford University Press, 2000.

\bibitem{DD}
J. Delaporte and A. Derighetti, Invariant projections and convolution operators, Proc. Amer. Math. Soc. \textbf{129} (2001), no. 5, 1427-1435

\bibitem{Der}
A. Derighetti, Conditional expectations on $CV\sb p(G)$. Applications, J. Funct. Anal. \textbf{247} (2007), no. 1, 231-251.

\bibitem{DP04}
H. G. Dales and M. E. Polyakov, Homological properties of modules over group algebras.  Proc. London Math. Soc. (3) \textbf{89} (2004), no. 2, 390-426.

\bibitem{ER00}
E. G. Effros and Z. J. Ruan, Operator spaces. London Mathematical Society Monographs. New Series, 23.
The Clarendon Press, Oxford University Press, New York, 2000.

\bibitem{ES92}
M. Enock, J.-M. Schwartz, Kac algebras and duality of locally compact groups. Springer-Verlag, Berlin, 1992.

\bibitem{Eymard}
P. Eymard, L'alg\`{e}bre de Fourier d'un groupe localement
compact., Bull. Soc. Math. France \textbf{92 }(1964), 181-236.

\bibitem{FK86}
T. Fack and H. Kosaki, Generalized $s$-numbers of $\tau$-measurable operators. Pacific J. Math. \textbf{123} (1986), no. 2, 269-300.

\bibitem{ForRun}  B. Forrest and V. Runde, Amenability and weak amenability of the Fourier algebra., 
Math. Z., \textbf{250} (2005), no. 4, pp. 731-744.

\bibitem{ForWood}
B. Forrest and P. Wood, Cohomology and the operator space structure of the 
Fourier algebra and its second dual., Indiana Univ. Math. J. \textbf{50} (2001), no. 3, 1217-1240. 
 
\bibitem{Hel}
A. Ya. Helemskii, The homology of Banach and topological algebras. Translated from the Russian by Alan West.
Mathematics and its Applications (Soviet Series), 41. Kluwer Academic Publishers Group, Dordrecht, 1989. xx+334 pp.

\bibitem{Iz97}
H. Izumi, Constructions of non-commutative $L\sp p$-spaces with a complex parameter arising from modular actions. 
Internat. J. Math. \textbf{8} (1997), no. 8, 1029-1066.

\bibitem{Iz98}
H. Izumi, Natural bilinear forms, natural sesquilinear forms and the associated duality on non-commutative $L\sp p$-spaces.
Internat. J. Math. \textbf{9} (1998), no. 8, 975-1039.

\bibitem{John}  B. Johnson, Cohomology in Banach algebras, Mem. Amer. Math.
Soc., \textbf{127} (1972).

\bibitem{JR}
M. Junge and Z. J. Ruan, Approximation properties for noncommutative $L\sb p$-spaces associated with discrete groups.
Duke Math. J. \textbf{117} (2003),  no. 2, 313-341. 

\bibitem{JRX05}
M. Junge, Z. J. Ruan and Q. Xu, Rigid $\mathcal{OL}\sb p$ structures of non-commutative $L\sb p$-spaces associated with
hyperfinite von Neumann algebras. Math. Scand. \textbf{96} (2005), no. 1, 63-95. 

\bibitem{JX03}
M. Junge and Q. Xu, Noncommutative Burkholder/Rosenthal inequalities. Ann. Probab. \textbf{31} (2003), no. 2, 948-995.

\bibitem{KL} E. Kaniuth and  A. T. Lau,
A Separation Property of Positive Definite Functions on Locally Compact Groups and Applications to Fourier Algebras,
J. Funct. Anal., \textbf{175}, n.1 (2000), 89-110.

\bibitem{KR83}
R. V. Kadison and J. R. Ringrose, Fundamentals of the theory of operator algebras. Vol. I. Elementary theory. Pure and Applied Mathematics, 100. 
Academic Press, Inc. [Harcourt Brace Jovanovich, Publishers], New York, 1983.

\bibitem{KR86}
R. V. Kadison and J. R. Ringrose, Fundamentals of the theory of operator algebras. Vol. II. Advanced theory. Pure and Applied Mathematics, 100.
Academic Press, Inc., Orlando, FL, 1986.

\bibitem{Ko84a}
H. Kosaki, Applications of the complex interpolation method to a von Neumann algebra: noncommutative $L\sp{p}$-spaces.
J. Funct. Anal. \textbf{56} (1984), no. 1, 29-78.

\bibitem{Li}
B.-R. Li, Introduction to Operator Algebras, World Scientific, 1992.

\bibitem{LL}
A. T. Lau and V. Losert,
Weak*-closed complemented invariant subspaces of $L^\infty(G)$ and amenable locally compact groups, Pacific J. Math. \textbf{123} (1986), 149-159.

\bibitem{LL90}
A. T. Lau and V. Losert,
Complementation of certain subspaces of $L\sb \infty(G)$ of a locally compact group.  Pacific J. Math. \textbf{141} (1990),  no. 2, 295-310.

\bibitem{Lep}  H. Leptin, Sur l'alg\`{e}bre de Fourier d'un groupe
localement compact., C.R. Acad. Sci. Paris S\'{e}r A \textbf{266 }(1968),%
1180-1182.

\bibitem{Pat}
A. L. T. Paterson, The class of locally compact groups $G$ for which $C\sp *(G)$ is amenable.
Harmonic analysis (Luxembourg, 1987),  226-237, Lecture Notes in Math., 1359, Springer, Berlin, 1988.

\bibitem{Pa01}
T. W. Palmer, Banach algebras and the general theory of $*$-algebras. Vol. 2. $*$-algebras.
Encyclopedia of Mathematics and its Applications, 79. Cambridge University Press, Cambridge, 2001.

\bibitem{P96}
G. Pisier, The operator Hilbert space ${\rm OH}$, complex interpolation and tensor norms. Mem. Amer. Math. Soc. \textbf{122} (1996), no. 585.

\bibitem{P98}
G. Pisier, Non-commutative vector valued $L_p$-spaces and completely $p$-summing maps. Ast\'{e}risque(Soc. Math. France)
\textbf{247} (1998), 1-111.

\bibitem{P03}
G. Pisier, Introduction to operator space theory. London Mathematical Society Lecture Note Series, 294. Cambridge
University Press, Cambridge, 2003.

\bibitem{Ru}  Z.-J. Ruan, The operator amenability of A(G), Amer. J.
Math., \textbf{117} (1995), 1449-1476.

\bibitem{RX97}
Z.-J. Ruan and G. Xu, Splitting properties of operator bimodules and operator amenability of Kac algebras. 
Operator theory, operator algebras and related topics (Timi\c soara, 1996), 193-216, Theta Found., Bucharest, 1997.

\bibitem{So62}
A. Sobczyk, Extension properties of Banach spaces. Bull. Amer. Math. Soc. \textbf{68} (1962) 217-224.

\bibitem{Stei}
H. Steiniger, Finite dimensional-extensions of Fourier algebras, preprint (1997).

\bibitem{Ta79}
M. Takesaki, Theory of operator algebras. I. Springer-Verlag, New York-Heidelberg, 1979.

\bibitem{Ta03}
M. Takesaki, Theory of operator algebras. II. Encyclopaedia of Mathematical Sciences, 125.
Operator Algebras and Non-commutative Geometry, 6. Springer-Verlag, Berlin, 2003.

\bibitem{Taylor} K. Taylor, Geometry of the Fourier algebras and locally compact groups with atomic unitary representations.
Math. Ann. \textbf{262} (1983), 183-190.

\bibitem{Te81}
M. Terp, $L^p$ spaces associated with von Neumann algebras, Notes, Math. Institute, Copenhagen Univ. 1981.

\bibitem{W02}
P. J. Wood, The operator biprojectivity of the Fourier algebra. Canad. J. Math. \textbf{54} (2002), no. 5, 1100-1120.

\bibitem{Z}
E. I. Zelmanov, On periodic compact groups, Israel J. Math. \textbf{77} (1992) 83-95.


\end{thebibliography}
\providecommand{\bysame}{\leavevmode\hbox to3em{\hrulefill}\thinspace}

\end{document}